\documentclass{article}
\usepackage{tikz}
\usepackage{tikz-cd}
\usepackage{graphicx}
\usepackage{pgfkeys}
\usetikzlibrary{arrows}
\usepackage[english]{babel}
\usepackage[utf8]{inputenc}
\usetikzlibrary{calc,babel}
\usepackage[all,cmtip]{xy}
\usepackage{extarrows}
\usepackage{amsfonts}
\usepackage{amsmath}
\usepackage[
backend=biber,
style=alphabetic,
]{biblatex}
\usepackage{stmaryrd}
\usepackage{mathtools}
\usepackage{amssymb}
\usepackage{amsthm}
\usepackage[hidelinks]{hyperref} 
\usepackage{mathtools}
\usepackage[capitalize]{cleveref}
\usepackage[shortlabels]{enumitem}

\usepackage[top=1.5in, bottom=1in, left=1in, right=1in]{geometry}
\usepackage{hyperref}
\usepackage{setspace}

\newtheorem{thm}{Theorem}[section]
\newtheorem{lem}[thm]{Lemma}

\newtheorem{cor}[thm]{Corollary}
\newtheorem{prop}[thm]{Proposition}

\newtheorem{const}[thm]{Construction}
\newtheorem{notn}[thm]{Notation}

\theoremstyle{definition}
\newtheorem{defi}[thm]{Definition}
\newtheorem{example}[thm]{Example}

\theoremstyle{remark}
\newtheorem{rem}{Remark}

\newcommand{\N}
{\mathbb N}

\newcommand{\Z}{\mathbb Z}

\newcommand{\h}{\mathcal{H}}

\newcommand{\kg}{{\mathcal{K}_G}}
\newcommand{\hatkg}{\widehat{\mathcal{K}_G}}
\newcommand{\R}{\mathbb R}

\newcommand{\GA}{G\text{-}\mathbb{A}^\nu}
\newcommand{\A}{\mathbb{A}^\nu}
\newcommand{\C}{\mathcal C}

\DeclarePairedDelimiterX\braket[2]{\langle}{\rangle}{#1 \delimsize\vert #2}
\DeclarePairedDelimiterX\inr[2]{\langle}{\rangle}{#1 \delimsize, #2}

\makeatletter
\newcommand{\colim@}[2]{%
  \vtop{\m@th\ialign{##\cr
    \hfil$#1\operator@font colim$\hfil\cr
    \noalign{\nointerlineskip\kern1.5\ex@}#2\cr
    \noalign{\nointerlineskip\kern-\ex@}\cr}}%
}
\newcommand{\colim}{%
  \mathop{\mathpalette\colim@{\rightarrowfill@\textstyle}}\nmlimits@
}
\makeatother

\title{Equivariant KK-theory and model categories}
\author{\scalebox{1.1}{Anupam Datta \thanks{Universität Bonn, anupamd@math.uni-bonn.de}, Michael Joachim\thanks{Universität Münster, joachim@uni-muenster.de}}}

\date{}
\begin{document}

\maketitle
\begin{abstract}
    We cast Kasparov's equivariant KK-theory in the framework of model categories. We obtain a stable model structure on a certain category of locally multiplicative convex $G$-$C^*$-algebras, which naturally contains the stable $\infty$-category $KK^G_{\operatorname{sep}}$ as described by Bunke, Engel, Land (\cite{BEL}). Non-equivariantly, $KK$-theory was studied using model categories by Joachim-Johnson (\cite{MJ}). We generalize their ideas in the equivariant case, and also fix some critical errors that their work had.
\end{abstract}
\tableofcontents
\newpage
\section{Introduction}
In this article, we study Kasparov's equivariant $KK$-theory from the perspective of model categories. The interest in this study is two-fold: firstly, off late a lot has been done in casting operator algebraic bivariant $K$-theories in the language of $\infty$-categories following the works of Bunke and coauthors; see \cite{BUNBEN},\cite{BEL},\cite{BD}. It is natural to wonder whether the stable $\infty$-category $KK^G_{\operatorname{sep}}$ constructed in \cite{BEL} can be realized as a full subcategory of the stable $\infty$-category underlying a stable model category. Although such model categories may be constructed as localizations of certain presheaf categories (see \cite[Proposition A.3.7.6]{HTT}), we construct a stable model structure on a certain category of topological $*$-algebras. This not only gives us a much more concrete category, but also a framework which the operator algebraists would appreciate. 

For the second motivation, we recall that we can view $KK$-theory as homotopy classes of certain equivariant quasihomomorphisms following Cuntz-Meyer. 
\begin{thm} {\cite[Theorem 5.5]{Mey}} \label{thm: Intro}
    Let $G$ be a locally compact second countable group, and let $A,B$ be $G$-C*-algebras\footnote{recall that this means the group acts via automorphisms on the $C^*$-algebra, and the action is continuous in the point-norm topology.}, with $A$ separable and $B$ $\sigma$-unital. Then, there is a natural isomorphism 
    \[
    KK^G(A,B) \cong [q(A \otimes \mathcal{K}_G)\otimes \kg, B \otimes \kg],
    \]
    where $q$ denotes the kernel of the fold map $A*A \rightarrow A$, and $\kg$ denote the $G$-$C^*$-algebra of compact operators on the $G$-Hilbert space $L^2G \otimes l^2\N$.
\end{thm}
It is plausible to ask whether $q(A \otimes \kg) \otimes \kg$ is a cofibrant replacement to $A$, and $B \otimes \kg$ is a fibrant replacement to $B$ for a certain model structure on $G$-$C^*$-algebras, whereby we would obtain \cref{thm: Intro} as a corollary to the fundamental theorem of model categories. We give a partial answer to this question in this paper, in the sense mentioned in the next paragraphs.

We now give an overview of the main results of this article. \textit{Throughout this article, we work with a locally compact, second countable group $G$.}\footnote{Although \cite{BEL},\cite{BUNBEN} is written for countable, discrete groups, the results from it that we use in this article can be readily verified to be true for general locally compact, second countable group actions.} We then consider the category of $\nu$-complete locally multiplicative convex (abbreviated lmc) $G$-$C^*$-algebras, where $\nu \gg 0$ is a cardinal (see \cref{defi: nu complete} for details). We denote this category by $\GA$. We can now state our main theorem.
\begin{thm}[=\cref{thm: diff formulation of main}]\label{thm: intro main}
   There is a simplicial model structure on the category of $\nu$-complete locally multiplicative convex $G$-$C^*$-algebras satisfying the following properties:
    \begin{enumerate}
        \item for a separable $G$-$C^*$-algebra $A$ and $B \in \GA$, there is a natural equivalence 
        \[
        \operatorname{Map}_{\GA[w^{-1}]}(A,B)\simeq KK^G(A,B)_\bullet,
        \]
        where the first entity stands for the mapping anima in the $\infty$-category obtained by inverting the weak equivalences of the model structure. In particular, when $A$,$B$ are separable $G$-$C^*$-algebras, we have a natural isomorphism
        \[
        \operatorname{Ho}(A,B) \cong KK^G_{0}(A,B).
        \]
        \item a $G$-equivariant $*$-homomorphism $A \rightarrow A'$ of separable $G$-$C^*$-algebras is a weak equivalence for the model category structure if and only if it is a $KK^G$-equivalence. It is a fibration if and only if the induced map \[\operatorname{Hom}(q(D\otimes \kg)\otimes \kg,A \otimes \kg)_\bullet \rightarrow \operatorname{Hom}(q(D \otimes \kg) \otimes \kg,A' \otimes \kg)_\bullet\] is a Kan fibration for separable $G$-$C^*$-algebras $D$.
        \item the model structure is cofibrantly generated.
        \item each object is fibrant.
        \item the model structure is stable, and the equivalences in Point 1 refine to equivalences of commutative groups.
    \end{enumerate}
\end{thm}
In order to motivate our category, we mention briefly some related work done in this realm. We stick to the non-equivariant case for simplicity. In \cite{UUY}, a stable category of fibrant objects structure on separable $C^*$-algebras is constructed, whose underlying stable $\infty$-category is naturally equivalent to the category $KK_{\operatorname{sep}}$ à la Bunke. The fibrations are the so-called \textit{Schochet fibrations}; a morphism $B \rightarrow D$ of separable $C^*$-algebras such that for every separable $C^*$-algebra $A$, the induced map 
\[
\operatorname{Hom}(A,B) \rightarrow \operatorname{Hom}(A,D)
\]
is a Serre fibration, where the hom sets are equipped with the point-norm topology. The weak equivalences are of course the $KK$-equivalences. In light of the work of Uuye, we could have hoped to construct a model structure on $C^*$-algebras by declaring a map $B \rightarrow D$ to be a fibration (resp. weak equivalence) if and only if the induced map 
\[
\operatorname{Hom}(qA,B\otimes\mathcal{K}(l^2)) \rightarrow \operatorname{Hom}(qA,D\otimes \mathcal{K}(l^2))
\]
is a fibration (resp. weak equivalence). But then, it would not be clear at all when an object is cofibrant; a $C^*$-algebra $A$ would be cofibrant if and only if we could find a lift in the diagram
\[
\begin{tikzcd}
	0 & B \\
	A & D
	\arrow[from=1-1, to=1-2]
	\arrow[from=1-1, to=2-1]
	\arrow["\simeq", two heads, from=1-2, to=2-2]
	\arrow[dashed, from=2-1, to=1-2]
	\arrow[from=2-1, to=2-2]
\end{tikzcd}
\]
for all acylic fibrations $B \rightarrow D$. This in general is a very hard problem, as experience shows that verifying a certain morphism of $C^*$-algebras is a Schochet fibration is either obvious, or very difficult. Therefore, it would also not be clear whether we could fulfill the functorial factorization axioms. However, if we could form a left adjoint to the functor $- \otimes \mathcal{K}(l^2)$, denoted by $- \boxtimes \mathcal{K}(l^2)$, and deem a morphism to be a fibration (resp. weak equivalence) if and only if the morphism
\[
\operatorname{Hom}(qA \boxtimes \mathcal{K}(l^2),B) \rightarrow \operatorname{Hom}(qA \boxtimes \mathcal{K}(l^2),D)
\]
is a fibration (resp. weak equivalence), then clearly (see \cref{cor: qA boxtimes hatkg is cofibrant}) we have a lift 
\[
\begin{tikzcd}
	0 & B \\
	{qA \boxtimes \mathcal{K}(l^2)} & D
	\arrow[from=1-1, to=1-2]
	\arrow[from=1-1, to=2-1]
	\arrow["\simeq", two heads, from=1-2, to=2-2]
	\arrow[dashed, from=2-1, to=1-2]
	\arrow[from=2-1, to=2-2]
\end{tikzcd}
\]
for every separable $C^*$-algebra $A$, and every acylic fibration $B \rightarrow D$, thereby making $qA \boxtimes \mathcal{K}(l^2)$ a cofibrant object. But, such left adjoints cannot be constructed in the category of $C^*$-algebras, as $- \otimes \mathcal{K}(l^2)$ does not preserve all limits\footnote{the reader may check that the natural map $(\prod_{n \in \N} \mathbb{C}) \otimes \mathcal{K}(l^2) \rightarrow \prod_{n \in \N}(\mathbb{C} \otimes \mathcal{K}(l^2))$ is not surjective, where the product is taken in the category of $C^*$-algebras.}. In the larger category of $\nu$-complete locally multplicative convex $C^*$-algebras, however, $- \otimes \mathcal{K}(l^2)$ indeed has a left adjoint; see \cref{prop: adj}.

In order to facilitate the construction of functorial factorizations of a morphism in $\GA$, we use the well-known \textit{Quillen's small object argument}, which requires that the algebras $qA \boxtimes \mathcal{K}(l^2)$ are \textit{small} whenever $A$ is a separable $C^*$-algebra. Since locally convex topologies are intrinsically a cofiltered limit topology, we cannot work with (Cauchy) complete algebras to fulfill the smallness requirement. To that end, we need to work with algebras where, roughly speaking, we can bound the length of the Cauchy nets required to form the completion. This calls for working with $\nu$-complete locally multiplicative convex $C^*$-algebras. Note that we do require some form of completeness so as to accommodate $C^*$-algebras within our category.

We furthermore mention that non-equivariantly, the problem of studying $KK$-theory using model categories had been addressed by Joachim-Johnson in \cite{MJ}. However, there were many technical issues in their work. The authors therein equipped the mapping sets between the algebras with the compact open topology, which simply fails at a technical level, as it does not render certain maps continuous. We remedy this by working with the more combinatorial category of simplicial sets. Moreover, the $\nu$-completion functor as described in loc. cit. is incorrect, and that makes us do some extra work in Section 7 to make the small object argument work. On top of that, we also show that our model category is stable, which was not addressed in loc. cit. This enables us to get a natural group structure on the homotopy category of our model category, which coincides with $KK(A,B)$ when $A$ and $B$ are separable $C^*$-algebras. Finally, as mentioned in the first paragraph, we remind the reader that $KK^G_{\operatorname{sep}}$ à la Bunke et. al. sits fully faithfully inside the stable $\infty$-category presented by our model category; see \cref{bunke fully faithful}.

We finally discuss the organization of this article. We begin by discussing general preliminaries on equivariant topological $*$-algebras in Section 2. We present sufficient details on the categorical aspect of these analytic objects, which is not readily available in the literature. In Section 3, we enrich the category of lmc $G$-$C^*$-algebras over simplicial sets. We next turn to the construction of universal $G$-equivariant $C^*$-algebras on a set of $G$-equivariant generators and relations in Section 4. This construction does not seem to have been studied before when the group is not discrete (for the discrete case, a reference is \cite{BUNBEN}). The construction could not have been carried out within the category of $G$-$C^*$-algebras, and we expect the construction to be of independent interest. Using this, we show that $\GA$ is tensored over simplicial sets, a feature which is also not possible in the category of $C^*$-algebras. In Section 5, we describe the crucial functor $- \boxtimes \mathcal{K}(\h)$, where $\h$ is a separable $G$-Hilbert space. In Section 6, we review some classical equivariant $KK$-theory. We also discuss some basic properties of Cuntz's $q$-construction in the generality of topological $*$-algebras. Using it, we formulate and study the concept of a left $KK^G$-equivalence between objects in $\GA$, which obviously restricts to the classical concept of $KK^G$ equivalence when restricted to separable $G$-$C^*$-algebras. The last three sections are devoted to finally construct the sought after model structures on $\GA$.
\\\\
\textit{Acknowledgements: The authors would like to thank Emma Brink, Ulrich Bunke, Benjamin Duenzinger, Johannes Ebert, Markus Hausmann, Matthias Lesch, Vikram Nadig, Maxime Ramzi, Stefan Schwede and Shervin Sorouri for discussions at several points which clarified some mathematical issues. The project as a whole was supported by the Deutsche Forschungsgemeinschaft
(DFG, German Research Foundation) under Germany's Excellence Strategy –
EXC-2047/1 – 390685813, Hausdorff Center for Mathematics, and EXC 2044 – 390685587, Mathematics M{\"u}nster: Dynamics–Geometry–Structure.} 
\section{Equivariant topological *-algebras}
In this section, we lay the background definitions and properties from the theory of (equivariant) topological *-algebras.
\subsection{Some categories of (equivariant) topological *-algebras}
We begin by recalling some standard definitions and properties from the non-equivariant theory. 
\begin{defi}
    A topological $*$-algebra $A$ is a (complex) algebra equipped with a vector space topology such that the multiplication $A \times A \rightarrow A$ is jointly continuous, together with a continuous involution $*:A \rightarrow A$. We denote by $\operatorname{Top}\!*\operatorname{-alg}$ the category whose objects are topological $*$-algebras and morphisms are continuous $*$-homomorphisms.
\end{defi}
\begin{rem}
    More succinctly, a topological-$*$-algebra is a $*$-algebra object in $\operatorname{Top}$, the category of all topological spaces and continuous maps. One could also consider the category of $*$-algebra objects internal to a convenient category of topological spaces, as described in \cite{DUBUC}, for example. But we don't consider this setup in this article, and are more in sync with classical functional analysis.
\end{rem}
\begin{defi}
    A $C^*$-seminorm on a $*$-algebra $A$ is a seminorm $p$ satisfying \[p(ab) \leq p(a)p(b),~ p(a^*)=p(a),~\text{and}~p(a^*a)=p(a)^2.\]
    for all $a,b \in A$. Given a $C^*$-seminorm $p$ on $A$, we denote by $A_p$ the $C^*$-algebra obtained as the Hausdorff completion of $A$ with respect to the seminorm $p$, i.e., $A_p:=\overline{A/\ker p}$. 
\end{defi}
\begin{defi}
    Let $A$ be a topological $*$-algebra. $A$ is called a locally multiplicative convex (abbreviated lmc) $C^*$-algebra if the topology on $A$ is induced by a family of $C^*$-seminorms. The set of all continuous $C^*$-seminorms on a lmc $C^*$-algebra will be denoted by $\mathcal{S}(A).$ We denote by $\operatorname{lmc}\!C^*\operatorname{-alg}$ the full subcategory of $\operatorname{Top}\!*\operatorname{-alg}$ containing locally multiplicative convex $C^*$-algebras.
\end{defi}
The following remarks are in order:
\begin{rem}
\hfill
\begin{itemize}
    \item In general, we do not demand any Hausdorff assumptions on our lmc $C^*$-algebras. However, most of the algebras (but not all, especially some algebras that come up in an intermediate construction, see \cref{thm: uni alg} for example) we will study in this article will be Hausdorff, and we will be accommodating them in the full subcategory of Hausdorff lmc $C^*$-algebras. $\mathbb{A}$ denotes the full subcategory of $\operatorname{lmc}\!C^*\operatorname{-alg}$ consisting of Hausdorff lmc $C^*$-algebras.
    \item We also do not require lmc $C^*$-algebras to be complete. However, we have to work with a weak form of completeness in this article, as defined below. We will only be interested in completion procedures on Hausdorff algebras.
    \item As mentioned in the Introduction, in order to facilitate some set theoretic arguments / bounds in our work, we cannot work with Cauchy complete algebras, but only with $\nu$-complete algebras (see \cref{defi: nu complete}), where $\nu$ is a cardinal such that $\lvert \nu \rvert \geq 2^{2^{\lvert \N \rvert}}$. Although in \cite{MJ}, the authors considered $\nu$-complete algebras with $\lvert \nu \rvert \geq \lvert \N \rvert$, we have the larger lower bound so as to have \cref{prop: CXB} available.
    \item From \cite[Theorem III.3.1]{MAL}, we see that each $A \in \mathbb{A}$ can be naturally embedded in the Cauchy complete lmc $C^*$-algebra \[\varprojlim_{p\in \mathcal{S}(A)} A_p \cong \overline{A}\]
\end{itemize}
\end{rem}
For the remainder of this article, we fix a cardinal $\nu$ such that $\lvert \nu \rvert \geq 2^{2^{\lvert \N \rvert}}$.
\begin{defi}\label{defi: nu complete}
    Let $A$ be a Hausdorff topological vector space. We say that $A$ is $\nu$-complete if every $\nu$-Cauchy sequence in $A$ converges. Recall that a $\nu$-sequence is a function $x:\kappa \rightarrow A$, where $\kappa \leq \nu$ is some ordinal. A $\nu$-sequence is called Cauchy if it has the property that for any open neighborhood $U$ of $0 \in A$, there exists $\gamma \in \kappa$ such that $~\forall~\alpha,\beta \geq \gamma, x_\alpha-x_\beta \in U$. 
    
    We will denote by $\mathbb{A}^{\nu}$ the full subcategory of $\mathbb{A}$ consisting of $\nu$-complete, Hausdorff lmc $C^*$-algebras. 
\end{defi}
\begin{defi}
    Let $A \in \mathbb{A}$. A $\nu$-completion of $A$ is a $\nu$-complete algebra $A'$, together with a dense embedding $\iota:A \rightarrow A'$ such that we have a natural isomorphism \[
    \operatorname{Hom}_{\mathbb{A}^\nu}(A',B) \xrightarrow[\cong]{\iota^*} \operatorname{Hom}_{\mathbb{A}}(A,B),
    \]
    for all $B \in \A$
\end{defi}
Since there isn't adequate literature coverage of $\nu$-complete algebras, we give a proof of the following result. The analogous result for Cauchy complete algebras is well known.
\begin{lem}\label{lem: nu completion construction}
   Each $A \in \mathbb{A}$ has a $\nu$-completion, which is unique upto isomorphism.
\end{lem}
\begin{proof}
    We just give a construction of a $\nu$-completion here; the verification of everything else is routine. One model for the $\nu$-completion of $A$ is the \textit{$\nu$-sequential closure} of $A$ inside $\overline{A}$, which can be described as follows: for any subset $S$ of $\overline{A}$, define $\nu\operatorname{-scl}(S)$ to be the set of limits in $A$ of $\nu$-sequences in $S$. Then by transfinite induction, we can define for a successor ordinal $\alpha+1$, \[\nu\operatorname{-scl}^{\alpha+1}(S):=\nu\operatorname{-scl}(\nu\operatorname{-scl}^{\alpha}(S)),\] and for a limit ordinal $\alpha$, \[\nu\operatorname{-scl}^{\alpha}(S):=\underset{\beta<\alpha}{\bigcup} \nu\operatorname{-scl}^{\beta}(S).\]
    The process will stabilize at a certain ordinal depending on $A$ \footnote{it will be crucial in \cref{sec: sep} that we have a better understanding of this ordinal for certain (equivariant) lmc $C^*$-algebras.} (see \cite{SEQCLOSURE}), yielding the sought after $\nu$-completion of $A$. The adjunction property can be verified directly from the construction.
\end{proof}
\begin{rem}
    We denote the $\nu$-completion of an algebra $A \in \mathbb{A}$ by $\overline{A}^\nu$. The Cauchy completion will be denoted (as usual) by $\overline{A}$.
\end{rem}
The key points about the categories described thus far are contained in the following proposition:

\begin{prop}\label{prop: non-eq cat}
    \hfill
    \begin{enumerate}
        \item $\operatorname{Top}\!*\operatorname{-alg}$ is bicomplete.
        \item There exists a left adjoint to the inclusion functor $\operatorname{lmc}\!C^*\operatorname{-alg} \rightarrow \operatorname{Top}\!*\operatorname{-alg}$, which we denote by $\operatorname{lmc-fication}$.
        \item There exists a left adjoint to the inclusion functor $\mathbb{A} \rightarrow \operatorname{lmc}\!C^*\operatorname{-alg}$, given by the Hausdorff quotient functor.
        \item There exists a left adjoint to the inclusion functor $\mathbb{A}^{\nu} \rightarrow \mathbb{A}$, given by the $\nu$-completion functor.
    \end{enumerate}
\end{prop}
This result is well known. We will prove an equivariant version of this in \cref{prop: eq cat}, and direct the reader to the proof of it for the details. Note that 1. of \cref{prop: non-eq cat} is proven in Lemma 2.4 and Lemma 2.5 in \cite{MJ}.

Summarizing \cref{prop: non-eq cat}, we get:
\begin{cor}\label{cor: non-eq cats}
    We have a sequence of adjunctions, where the left adjoints go from left to right, and all the right adjoints are fully faithful:
        \[
\begin{tikzcd}
	{\operatorname{Top}\!*\operatorname{-alg}} && {\operatorname{lmc}\!C^*\operatorname{-alg}} && {\mathbb{A}} && {\mathbb{A}^{\nu}}
	\arrow["{\operatorname{lmc-fication}}", from=1-1, to=1-3]
	\arrow[bend left, hook, from=1-3, to=1-1]
	\arrow["{\operatorname{Haus. quotient}}", from=1-3, to=1-5]
	\arrow[bend left, hook, from=1-5, to=1-3]
	\arrow["{\nu-\operatorname{completion}}", from=1-5, to=1-7]
	\arrow[bend left, hook, from=1-7, to=1-5]
\end{tikzcd}
        \]
In particular, each of the categories mentioned above are bicomplete.
\end{cor}
\begin{proof}
    This follows from a general result in category theory; see \cite[Proposition~7.21]{BUNKE3} for example. We mention the recipe of computing limits and colimits here for the ease of reference later. Each of the inclusion functors preserve and reflect limits\footnote{note that in loc. cit., its only mentioned that the right adjoint detects limits. The present assertion follows from the proof of it; see also \cite[Section 3]{CONS}.}. If \[
\begin{tikzcd}
	\C & \mathcal{D}
	\arrow["L", from=1-1, to=1-2]
	\arrow["R", shift left=3, hook, from=1-2, to=1-1]
\end{tikzcd}
    \]
    is a typical adjunction above, then the colimit of a diagram $D:I \rightarrow \mathcal{D}$ is given by 
    \[
    \underset{I}{\operatorname{colim}}D \cong L(\underset{I}{\operatorname{colim}}R(D)).
    \]
\end{proof}
\begin{rem}\label{rem: lim top non-eq}
    The fact that the underlying space of a limit of topological $*$-algebras is the limit of the corresponding underlying spaces is a key fact for many arguments in this article. See \cite[III, Lemma 2.1]{MAL} for the details.
\end{rem}
\begin{rem}\label{rem: colimit seminorm}
    For later uses, we present here another description of colimits in $\operatorname{lmc}\!C^*\operatorname{-alg}$. If $F:I \rightarrow \operatorname{lmc}\!C^*\operatorname{-alg}$ is any diagram, then $\operatorname{colim}_{\operatorname{lmc}\!C^*\operatorname{-alg}}F$ is the *-algebra $\operatorname{colim}_{\operatorname{*-alg}}U \circ F$, (where 
    $
    U:\operatorname{lmc}\!C^*\operatorname{-alg}\rightarrow \operatorname{*-alg}$
    is the forgetful functor), equipped with the following collection $\mathcal{S}$ of $C^*$-seminorms:
    \[
    \text{A $C^*$-seminorm $p \in \mathcal{S}$ $\iff$ $p \circ f_i \in \mathcal{S}(F(i))$ $\forall ~i \in I$},
    \]
    where $f_i:F(i) \rightarrow \operatorname{colim}_{\operatorname{*-alg}}U \circ F$ are the bonding maps of the universal cocone. The fact that it is indeed the colimit in $\operatorname{lmc}\!C^*\operatorname{-alg}$ follows by a direct verfication of the relevant universal property. 
\end{rem}
\begin{rem}
    One can construct limits and colimits in $\operatorname{Top}\!*\operatorname{-alg}$ from limits and colimits in $\operatorname{*-alg}$, using ideas similar to the one mentioned in \cref{rem: colimit seminorm}. This could be accomplished by first establishing both left and right adjoint to the forgetful functor $U:\operatorname{Top}\!*\operatorname{-alg}\rightarrow \operatorname{*-alg}$, and arguing that $U$ lift both limits and colimits. We, however, don't take that route\footnote{the reader who wants to tread along this path may find \cite[Chapter 6]{JOY} interesting!}, and consider topological $*$-algebras as primitive objects in our work.
\end{rem}
Now, we come to the equivariant definitions. Recall that, in this article, $G$ denotes a locally compact, second countable group.
\begin{defi}\label{defi: G-action}
     A $G$-algebra $A$ is a topological $*$-algebra $A$ together with an action of $G$ on $A$, i.e., a homomorphism from $G$ to $\operatorname{Aut}(A)$ (the set of all algebraic $*$-automorphisms of $A$ which are continuous and have continuous inverse) such that for each $a \in A$, the map $g \mapsto g.a$ is continuous. \textit{Throughout the article, we will only consider equivariant continuous $*$-homomorphisms, often without the adjective ``equivariant", between $G$-algebras.}

     The category of $G$-algebras and equivariant continuous $*$-homomorphisms will be denoted by $G\text{-}\!\operatorname{Top}\!*\operatorname{-alg}$. The hom set from algebras $A$ to $B$ will be denoted by $\operatorname{Hom}(A,B)$.
\end{defi}
\begin{defi}
A $C^*$-seminorm $p$ on a G-algebra $A$ is called $G$-invariant if 
\[
    p(g.a)=p(a) ~\forall~g \in G, ~\forall~a \in A.
\]
\end{defi}
\begin{defi}\label{defi: eq lmc G alg}
    A $G$-algebra $A$ is called a lmc $G$-$C^*$-algebra if the topology on $A$ is induced by a family of $G$-invariant $C^*$-seminorms. We abuse notation, and in the equivariant context, denote the set of all $G$-invariant continuous $C^*$-seminorms on $A$ as $\mathcal{S}(A)$. The full subcategory of $G\text{-}\!\operatorname{Top}\!*\operatorname{-alg}$ consisting of lmc $G$-$C^*$-algebras will be denoted by $G\text{-}\!\operatorname{lmc}\!C^*\operatorname{-alg}$.
\end{defi}
\begin{rem}
    It is crucial to assume in \cref{defi: eq lmc G alg} that the algebra had a continuous $G$-action to begin with. For example, consider $\C(S^1)$ with the actions
    \begin{align*}
        \alpha_n: S^1 &\times \C(S^1) \rightarrow \C(S^1) \\
        (u,&f) \longmapsto v \mapsto f(u^nv)        
    \end{align*}
    Consider the product 
    \[
    A := \prod_{n \in \N} \C(S^1,\alpha_n)
    \]
    in the category of $C^*$-algebras. It can be checked that the action of $S^1$ on $A$ is not continuous. Nevertheless, the $C^*$-norm on $A$ is $G$-invariant.
\end{rem}
The following corollary is an equivariant analogue of \cite[Lemma 1.1.5]{Phi}.
\begin{cor}\label{cor: factor}
    Let $A$ be a lmc $G$-$C^*$-algebra and $B$ a $G\text{-}C^*$-algebra. Then, any continuous $*$-homomorphism $\varphi: A \rightarrow B$ factors through a $G\text{-}C^*$-quotient of $A$.
\end{cor}
\begin{proof}
    Follows easily from the fact that $a \mapsto \lVert \varphi(a) \rVert$ is a continuous $G$-invariant C*-seminorm on $A$.
\end{proof}
\begin{defi}
    The full subcategory of Hausdorff algebras in $G\text{-}\!\operatorname{lmc}\!C^*\operatorname{-alg}$ will be denoted by $G\text{-}\mathbb{A}$. The full subcategory of $G\text{-}\mathbb{A}$ comprising of $\nu$-complete algebras will be denoted by $\GA$.
\end{defi}
We now mention the equivariant analogue of \cref{prop: non-eq cat}:
\begin{prop}\label{prop: eq cat}
    \hfill
    \begin{enumerate}
        \item $G\text{-}\!\operatorname{Top}\!*\operatorname{-alg}$ is bicomplete.
        \item There exists a left adjoint to the inclusion functor $G\text{-}\!\operatorname{lmc}\!C^*\operatorname{-alg} \rightarrow G\text{-}\!\operatorname{Top}\!*\operatorname{-alg}$, which we denote by $\operatorname{lmc-fication}$.
        \item There exists a left adjoint to the inclusion functor $G\text{-}\mathbb{A} \rightarrow G\text{-}\!\operatorname{lmc}\!C^*\operatorname{-alg} $, given by the Hausdorff quotient functor.
        \item There exists a left adjoint to the inclusion functor $G\text{-}\mathbb{A} \rightarrow \GA$, given by the $\nu$-completion functor.
    \end{enumerate}
\end{prop}
\begin{proof}
    \hfill
    \begin{enumerate}
        \item Limits are computed in $\operatorname{Top*-alg}$, and the group action is given coordinatewise. The continuity of the group action follows from the universal property of limits in $\operatorname{Top}$, and \cref{rem: lim top non-eq}.

        Colimits are also computed in $\operatorname{Top*-alg}$, though the argument is not as straightforward. By basic category theory, it suffices to show that finite coproducts, coequalizers, and directed colimits exist in $G\text{-}\!\operatorname{Top*-alg}$. In fact, we show that the canonical group action on the respective construction in $\operatorname{Top-*-alg}$ is continuous. We just discuss the case of finite coproducts here; the other cases are very straightforward.
        
        Let $\{A_i\}_{1\leq i \leq n}$ be a finite collection of objects in $G\text{-}\!\operatorname{Top*-alg}$. Recall that a (finite) coproduct in $\operatorname{Top}\!*\operatorname{-alg}$ is obtained by equipping the $*$-algebraic free product with a certain topology; see \cite[Lemma 2.5]{MJ}. The group action on $A_1*A_2*...*A_n$ is given by setting \[g.(a_1*a_2*...*a_n):=(ga_1)*(ga_2)*...*(ga_n),\] and extending linearly. The continuity of the group action is immediate from the fact that multiplication is jointly continuous in topological $*$-algebras.
        \item We just define the functor here; the adjunction property will then follow easily. Given $A \in G\text{-}\!\operatorname{Top}\!*\operatorname{-alg} $, let $(A,\tau_{\operatorname{lmc}})$ denote the $*$-algebra $A$ with the action of (the underlying group of) $G$, equipped with the topology given by 
        \[
        a_i \rightarrow a ~\iff~ p(a_i - a) \rightarrow 0 ~\forall~ p \in \mathcal{S}(A).
        \]
        It is easy to check that $(A,\tau_{\operatorname{lmc}})$ is a lmc $G\text{-}C^*$-algebra, and a ($G$-equivariant) continuous $*$-homomorphism $A \rightarrow B$ induces a ($G$-equivariant) continuous $*$-homomorphism $(A,\tau_{\operatorname{lmc}}) \rightarrow (B,\tau_{\operatorname{lmc}})$, which yields the lmc-fication functor.
        \item Let $A \in G\text{-}\!\operatorname{lmc}\!C^*\operatorname{-alg}$, and let $\{p_\lambda\}_{\lambda \in \Lambda}$ be a family of $G$-invariant C*-seminorms inducing the topology on $A$.  We have \[A^{\operatorname{Haus}}=A/\overline{\{0\}}.\] Clearly, each $p_\lambda$ descends to a continuous $G$-invariant C*-seminorm on $A^{\operatorname{Haus}}$. These $\{p_\lambda\}_{\lambda \in \Lambda}$  in fact also generate the topology on $A^{\operatorname{Haus}}$, as can be seen from \cref{rem: colimit seminorm}, for example.
        \item The construction of \cref{lem: nu completion construction} works verbatim in the equivariant context. The (continuous) group action on the $\nu$-completion will be inherited from the (continuous) group action on the Cauchy completion; see \cref{lem: G-AM light}.
    \end{enumerate} 
\end{proof}
We also mention the equivariant analogue of \cref{cor: non-eq cats}.
\begin{cor}\label{cor: eq bicom}
    We have a sequence of adjunctions, where the left adjoints go from left to right, and all the right adjoints are fully faithful:
        \[
\begin{tikzcd}
	{G\text{-}\!\operatorname{Top*-Alg}} && {G\text{-}\!\operatorname{lmc C*-alg}} && {G\text{-}\mathbb{A}} && {G\text{-}\mathbb{A}^{\nu}}
	\arrow["{\operatorname{lmc-fication}}", from=1-1, to=1-3]
	\arrow[bend left, hook, from=1-3, to=1-1]
	\arrow["{\operatorname{Haus. quotient}}", from=1-3, to=1-5]
	\arrow[bend left, hook, from=1-5, to=1-3]
	\arrow["{\nu-\operatorname{completion}}", from=1-5, to=1-7]
	\arrow[bend left, hook, from=1-7, to=1-5]
\end{tikzcd}
        \]
    Thus, each of the categories mentioned above are bicomplete.
\end{cor}
Finally, we note that limits and colimits in the equivariant categories mentioned before compare with those in the underlying non-equivariant categories. The proof is evident from \cref{prop: eq cat}, and from the  description of limits and colimits as mentioned in the proof of \cref{cor: non-eq cats}.
\begin{cor}\label{cor: forget preserve lim/colim}
    In the following diagram, every square that should commute does commute:
    \[
\begin{tikzcd}
	{G\text{-}\!\operatorname{Top*-Alg}} && {G\text{-}\!\operatorname{lmc C*-alg}} && {G\text{-}\mathbb{A}} && {G\text{-}\mathbb{A}^{\nu}} \\
	\\
	{\operatorname{Top*-Alg}} && {\operatorname{lmc C*-alg}} && {\mathbb{A}} && {\mathbb{A}^{\nu}}
	\arrow["{\operatorname{lmc-fication}}", from=1-1, to=1-3]
	\arrow["{\operatorname{forget}}", from=1-1, to=3-1]
	\arrow[bend left, hook, from=1-3, to=1-1]
	\arrow["{\operatorname{Haus. quotient}}", from=1-3, to=1-5]
	\arrow["{\operatorname{forget}}", from=1-3, to=3-3]
	\arrow[bend left, hook, from=1-5, to=1-3]
	\arrow["{\nu-\operatorname{completion}}", from=1-5, to=1-7]
	\arrow["{\operatorname{forget}}", from=1-5, to=3-5]
	\arrow[bend left, hook, from=1-7, to=1-5]
	\arrow["{\operatorname{forget}}", from=1-7, to=3-7]
	\arrow["{\operatorname{lmc-fication}}", from=3-1, to=3-3]
	\arrow[bend left, hook, from=3-3, to=3-1]
	\arrow["{\operatorname{Haus. quotient}}", from=3-3, to=3-5]
	\arrow[bend left, hook, from=3-5, to=3-3]
	\arrow["{\nu-\operatorname{completion}}", from=3-5, to=3-7]
	\arrow[bend left, hook, from=3-7, to=3-5]
\end{tikzcd}\]
Thus, each of the forgetful functors above preserve and reflect limits and colimits.
\end{cor}
\begin{rem} \label{rem: Hausdorff nu completion}
    In the sequel, we will refer to the functor
    \[
    (\nu-\operatorname{completion}) \circ (\operatorname{Haus. quotient}): G\text{-}\!\operatorname{lmc}\!C^*\operatorname{-alg} \rightarrow \GA
    \]
    as the Hausdorff $\nu$-completion.
\end{rem}
The following result is the equivariant analogue of the Arens-Michael decomposition theorem, compare \cite[Theorem 2.6]{MJ} and the references therein. The proof in the non-equivariant case carries over verbatim here:
\begin{lem} \label{lem: G-AM light}
    Let $A$ be a Hausdorff lmc. $G\text{-}C^*$-algebra. Then, there is a natural inclusion 
    \[
    A \subseteq \varprojlim\limits_{p \in \mathcal{S}(A)} A_p.
    \]
    Conversely, any $G$-invariant subalgebra of an inverse limit (in $G\text{-}\!\operatorname{Top*-Alg}$) of $G\text{-}C^*$-algebras is a Hausdorff lmc $G\text{-}C^*$-algebra
\end{lem}
\begin{cor}
    The group action on a Hausdorff lmc $G\text{-}C^*$-algebra $A$ is jointly continuous, that is, the map
    \begin{align*}
        G \times A &\rightarrow A\\
        (g,a) &\mapsto g.a
    \end{align*}
    is continuous.
\end{cor}
\begin{proof}
    In light of \cref{lem: G-AM light}, it suffices to show the assertion for an inverse limit of $G\text{-}C^*$-algebras (in $G\text{-}\!\operatorname{Top*-alg}$). But this is clear from the universal properties of inverse limits in $\operatorname{Top}$ (cf. \cref{rem: lim top non-eq}), and the fact that pointwise continuous group action on a $C^*$-algebra is in fact jointly continuous (since $C^*$-algebra (auto)morphisms are contractive, one can use \cite[Lemma 2.19]{D}).
\end{proof}
\subsection{Tensor products and function algebras}
We now discuss tensor products in the setting of lmc $G$-$C^*$-algebras.
\begin{defi}\label{defi: equi tensor prod}
    Let $A,B$ be lmc $G$-$C^*$-algebras. Let $A \odot B$ denote the algebraic tensor product of $A$ and $B$, together with the diagonal $G$-action. We equip $A \odot B$ with the initial topology given by the map
    \[
    A\odot B  \longrightarrow \varprojlim_{p \in \mathcal{S}(A), q \in \mathcal{S}(B)} {A_p} \otimes_G {B_q},
    \]
    where $A_p \otimes_G B_q$ denotes the \textit{maximal}\footnote{the choice of the tensor product is not very crucial at this juncture. But we decided to work with the maximal tensor product throughout because it has better functorial properties.} tensor product of the $G$-$C^*$-algebras $A_p$ and $B_q$ equipped with the diagonal $G$-action.

    Then, $A \otimes_G B$ denotes $A \odot B$ with the aforeprescribed topology.
\end{defi}
\begin{rem}
    Clearly, \cref{defi: equi tensor prod} is well-defined upon restricting to cofinal subsets of $\mathcal{S}(A)$ and $\mathcal{S}(B)$.
\end{rem}
That the tensor product of two continuous $*$-homomorphisms is again a continuous $*$-homomorphism follows from \cref{cor: factor}, whence we conclude:
\begin{cor}\label{cor: ten pro}
    The tensor product as defined above is functorial, i.e., we have a functor
    \[
    G\text{-}\!\operatorname{lmc C*-alg} \times G\text{-}\!\operatorname{lmc C*-alg} \xrightarrow{-\otimes_G-} G\text{-}\!\operatorname{lmc C*-alg}.
    \]
\end{cor}
Since $G$-equivariant continuous seminorms are cofinal in the set of all continuous seminorms on a lmc $G\text{-}C^*$-algebra, we conclude
\begin{cor}\label{cor: forget tensor}
    The following square commutes\footnote{the non-equivariant tensor product can be defined by taking $G$ to be the trivial group in \cref{defi: equi tensor prod}}:
    \[
\begin{tikzcd}
	{G\text{-}\!\operatorname{lmc C*-alg} \times G\text{-}\!\operatorname{lmc C*-alg}} && {G\text{-}\!\operatorname{lmc C*-alg}} \\
	{\operatorname{lmc C*-alg} \times \operatorname{lmc C*-alg}} && {\operatorname{lmc C*-alg}}
	\arrow["{-\otimes_G-}", from=1-1, to=1-3]
	\arrow["{\operatorname{forget} \times \operatorname{forget}}"', from=1-1, to=2-1]
	\arrow["{\operatorname{forget}}", from=1-3, to=2-3]
	\arrow["{-\otimes-}", from=2-1, to=2-3]
\end{tikzcd}
    \]
    Thus, the equivariant tensor product of lmc $G\text{-}C^*$-algebras is (uniquely) determined by the tensor product of the underlying lmc $C^*$-algebras, with the diagonal $G$-action.
\end{cor}
In light of the preceeding corollary, we will henceforth denote the tensor product in the equivariant setting also by $-\otimes-$.
\begin{defi}
    Let $A,B \in \GA$. Then, the tensor product of $A$ and $B$ (in $\GA$) is defined to be the Hausdorff $\nu$-completion of $A \otimes B$ from \cref{defi: equi tensor prod}. We will abuse notation and also denote the tensor product in $\GA$ by $- \otimes -$.
\end{defi}
We end this section with a discussion on function algebras in $\GA$. These will be ubiquitous for us while discussing the relation of homotopy in the next section. We first recapitulate the definition of the topology of uniform convergence in the generality of locally convex spaces. See \cite[Section 3.3]{TVS2} for more details and proofs.
\begin{defi}
    Let $X$ be a compact Hausdorff space, and $A \in G\text{-}\mathbb{A}$. The topology of uniform convergence on $\mathcal{C}(X,A)$ is the unique translation invariant topology on the set of continuous functions from $X$ to $A$, such that the collection of sets 
    \begin{equation}\label{eqn: nbhd base}
    \mathcal{U}_{\varepsilon,p_1,p_2,...,p_k}:=\{g \in \mathcal{C}(X,A) \mid p_i^X(g)< \varepsilon ~\forall~1 \leq i \leq k\}
    \end{equation}
    forms a neighbourhood base of $0$, where $\varepsilon > 0, k \in \N$, and $p_1,p_2,...,p_k \in \mathcal{S}(A)$, and for any $p \in \mathcal{S}(A)$, we let
    \[
    p^X(f):=\underset{x \in X}{\operatorname{sup}}~p(f(x)).
    \] 
\end{defi}

\begin{prop} \label{prop: CXB}
    Let $X$ be a compact Hausdorff space. If $B \in G\text{-}\mathbb{A}$, then $\mathcal{C}(X,B) \in G\text{-}\mathbb{A}$. If $B \in \GA$, then $\mathcal{C}(X,B) \in G\text{-}\A$. When $X$ is furthermore metrizable, $\C(X) \otimes B \cong \C(X,B)$.
\end{prop}  
\begin{proof}
    That $\mathcal{C}(X,B)$ is a locally multiplicative convex $C^*$-algebra with the collection $\{p^X \mid p \in \mathcal{S}(B)\}$ being a defining family of seminorms follows from \cite[Section 3.3]{TVS2}.  Note that $\C(X,B)$ carries a continuous $G$-action from \cref{lem: G-AM light}, and the fact that $\C(X,B)_{p^X} \cong \C(X,B_p)$ are $G$-$C^*$-algebras. The $\nu$-completeness of $\C(X,B)$ will follow from the $\nu$-completeness of $B$.

    To show the statement about tensor products, we will use some well known facts from general topology, for which we refer to \cite{MUN}. We first show that $\C(X) \odot B$ is dense in $\C(X,B)$. Let
    \[
    U_{g,\varepsilon,p_1,p_2,...,p_k}:=\{\tilde{g} \in \C(X,B) \mid p_i^X(g-\tilde{g})< \varepsilon ~\forall~1 \leq i \leq n \}.
    \]
    be an arbitrary (but fixed) basic open set in $\C(X,B)$. We want to show that it contains some element of $\C(X) \odot B$. Consider $g(X) \subseteq B$. From the compactness of $X$, we conclude that 
    \[
    g(X) \subseteq \bigcup_{i=1}^{n}U_{g(x_i),\varepsilon,p_1,p_2,...,p_k}, ~\text{where $x_1,x_2,...,x_n$ are points in X}.
    \]
    Set 
    \begin{equation}\label{eqn: open set}
        U_i:=g^{-1}U_{g(x_i),\varepsilon,p_1,p_2,...,p_k},
    \end{equation}
    and note that $X=\bigcup_{i=1}^nU_i$. Choose a partition of unity $\{h_i\}_{i=1}^{n}$ subordinate to this cover, such that $\operatorname{supp}(h_i) \subseteq U_i$. Then, for any $x \in X$, and for any $l \in \{1,...,k\}$, we calculate:
    \begin{align*}
        p_l(g(x)-\sum_{i=1}^{n}h_j(x)g(x_j))
        =p_l(\sum_{i=1}^{n}h_j(x)g(x)-\sum_{i=1}^{n}h_j(x)g(x_j)) 
        \leq \sum_{i=1}^{n}h_j(x)p_l(g(x)-g(x_j)) <\sum_{j=1}^{n}h_j(x)\varepsilon
        =\varepsilon,
    \end{align*}
where the last inequality follows from the definition of $U_i$ in \cref{eqn: open set}. Thus, the function
\begin{equation}\label{eqn: C(X).B}
    x \mapsto \sum_{i=1}^{n}h_j(x)g(x_j).
\end{equation}
lies in $U_{g,\varepsilon,p_1,p_2,...,p_k}$, showing that $\C(X) \odot B$ is dense in $\C(X,B)$.

    We now show that, when $X$ is metrizable, given $g \in \C(X,B),$ there exists a $\nu$-Cauchy sequence in $\C(X) \odot B$ converging to $g$. To that end, it suffices to show that there is a $\nu$-Cauchy sequence in $\C(X) \odot \langle \operatorname{im} g \rangle$ converging to $g$ in $\C(X,\langle \operatorname{im}g \rangle) \hookrightarrow \C(X,B)$, where $\langle \operatorname{im} g \rangle$ denotes the $*$-subalgebra of $B$ generated by $\operatorname{im}g$, equipped with the subspace topology. Note that, 
\begin{align*}
    \lvert \operatorname{im} g \rvert \leq \lvert X \rvert \leq \lvert \R \rvert 
    \implies \lvert \mathcal{S}(\langle \operatorname{im} g \rangle) \rvert \leq \lvert \R \rvert^{\lvert \R \rvert }=2^{2^{\lvert \N \rvert}},
\end{align*} 
which shows, in light of \cref{eqn: nbhd base}, that $\C(X,\langle \operatorname{im}g \rangle)$ has a neighborhood base of $g$ of cardinality less than or equal to $2^{2^{\lvert \N \rvert}}$. From the proof of the first part of this proposition, especially \cref{eqn: C(X).B}, we conclude that there exists a $\nu$-Cauchy sequence in $\C(X) \odot \langle \operatorname{im}g \rangle$ which converges to $g$.
\end{proof}
\begin{rem}
    We have an analogous result to \cref{prop: CXB} in the setting of pointed spaces, which we don't spell out here. That can be obtained by focusing on the ideal of functions vanishing at the base point.
\end{rem}
We end this section by a technical lemma, which will come handy in the next section.
\begin{lem}\label{lem: pushout to pullback}
    For each $B$ in $\GA$, the functor $\C(-,B)$ sends a pushout square of closed inclusions of compact Hausdorff spaces to a pullback square in $\GA$.
\end{lem}
\begin{proof}
    Let $X_{11}=X_{01}\coprod_{X_{00}}X_{10}$ be a pushout in compact Hausdorff spaces as mentioned in the statement. Then, by the universal property of colimits, we get a pullback of \textit{sets}
    \begin{equation}\label{eqn: pullback}
         \C(X_{11},B) \cong \C(X_{01},B) \times_{\C(X_{00},B)}\C(X_{10},B).
    \end{equation}
    By the explicit description of pullbacks in sets, we conclude that the pullback in \cref{eqn: pullback} is a pullback of \textit{$*$-algebras}. We next note that it is also a pullback in \textit{topological spaces}. Indeed, we have a natural continuous bijection 
    \[
    \C(X_{11},B) \xrightarrow{} \C(X_{01},B) \times^{\operatorname{Top}}_{\C(X_{00},B)}\C(X_{10},B).
    \]
    That its a homeomorphism follows by the explicit definition of convergence of a net in the uniform topology. This shows that \cref{eqn: pullback} is a pullback in $\operatorname{Top}\!*\operatorname{-alg}$, see \cite[Lemma III.2.1]{MAL}. That its in fact a pullback in $\GA$ follows from \cref{cor: forget preserve lim/colim} and \cref{cor: eq bicom}.   
\end{proof}
\section{Hom sets as Kan complexes}
In this section, we deal with the problem of ``topologizing" the hom sets between objects of $\GA$. Some of the obvious candidates like the compact-open topology and the point-norm topology (see \cite{TVS2} for the definition of these topologies, as well as other natural topologies on spaces of continuous linear maps between locally convex spaces) do not work for us at a technical level. \footnote{The compact open topology was used in \cite{MJ}, but it does not render all the maps we want to be continuous as continuous. In particular, the tensor product functor of \cref{cor: ten pro} isn't continuous, and there is a mistake in loc. cit. wherein the authors assumed it!}

We instead want to make $\GA$ enriched in Kan complexes, and that resolves all the technical issues. This approach has been taken for $C^*$-algebras before; see \cite[Section 3]{MAHANTA}.
\begin{defi}
    Let $A,B \in \GA$. We upgrade the set of $G$-equivariant $*$-homomorphisms between $A$ and $B$ into a simplicial set $\operatorname{Hom}(A,B)_\bullet$, whose $n$-simplices are defined by
    \[
    \operatorname{Hom}(A,B)_n:=\operatorname{Hom}(A,\C(\lvert \Delta^n \rvert,B)),
    \]
    with the obvious face and degeneracy maps.
\end{defi}
The following result is the key to the rest of the work.
\begin{prop}
    For $A,B \in \GA$, $\operatorname{Hom}(A,B)_\bullet$ is a Kan complex.
\end{prop}
\begin{proof}
    Since $\lvert \Lambda^n_k\rvert$ is a retract of $\lvert \Delta^n \rvert$ for all $n \in \N, 0 \leq i \leq n$, it suffices to show that there is a natural bijection
    \begin{equation}\label{eqn: horn}
        \operatorname{Fun}(\Lambda^n_i,\operatorname{Hom}(A,B)_\bullet) \cong \operatorname{Hom}(A,\C(\lvert \Lambda^n_i\rvert,B)).
    \end{equation}
    Recall that every horn $\Lambda^n_i$ can be expressed as a coequalizer (see \cite[I.3.1]{GJ} for the details) as follows: \[
    \coprod_{0\leq i <j \leq n} \Delta^{n-2} \rightrightarrows \coprod_{i \neq k} \Delta^{n-1} \rightarrow \Lambda^n_k.
    \]
    Thus, by the universal property of (co)limits and the Yoneda lemma, we have a natural isomorphism
    \[
    \operatorname{Fun}(\Lambda^n_i,\operatorname{Hom}(A,B)_\bullet) \cong \operatorname{Hom}(A, \varprojlim( \coprod_{i \neq k} \C(\lvert \Delta^{n-1} \rvert,B) \rightrightarrows \coprod_{0\leq i < j \leq n} \C(\lvert \Delta^{n-2} \rvert,B))).
    \]
    By \cref{lem: pushout to pullback} and \cite[Proposition III.13.3]{JOY}, we conclude that \cref{eqn: horn} holds, as desired. 
\end{proof}
\begin{rem}\label{rem: cpct open}
    If $A$ is a $G$-$C^*$-algebra, then for all $B \in \GA$, $\operatorname{Hom}(A,B)_\bullet$ is naturally isomorphic to $\operatorname{Sing}(\underline{\operatorname{Hom}}(A,B))$, where $\underline{\operatorname{Hom}}(A,B)$ denotes the set of $G$-equivariant continuous $*$-homomorphisms from $A$ to $B$, equipped with the compact open topology. This follows from \cite[Proposition 3.4]{MJ}, for example.
\end{rem}
The key to working with the Kan-enriched setup is that it makes the following result trivial.
\begin{prop}\label{prop: Kan enriched}\label{prop: ten pro simp enrichment}
    Let $A,B,C \in \GA$. Then, we have a natural morphism of simplicial sets
    \[
    \operatorname{Hom}(A,B)_\bullet \times \operatorname{Hom}(B,C)_\bullet \rightarrow \operatorname{Hom}(A,C)_\bullet,
    \]
    whereby making $\GA$ into a Kan enriched category. Furthermore, the tensor product functor described in \cref{cor: ten pro} upgrades to a simplicially enriched functor.
\end{prop}
We now show a fundamental property of our Kan enriched mapping spaces, to the effect that they are cotensored over simplicial sets.
\begin{prop}\label{prop: tensored and cotensored}
    Let $A,B \in \GA$, and $K$ a finite simplicial set. Then, we have a natural isomorphism
    \[
    \operatorname{Hom}(A,C(|K|,B))_\bullet \cong \operatorname{Map}_{\operatorname{sSet}}(K,\operatorname{Hom}(A,B)_\bullet)
    \]
\end{prop}
\begin{proof}
    When $A$ and $B$ are both $G$-$C^*$-algebras, the statement is well known (see \cite{MAHANTA}, \cite{MJ}, \cite{D}). When $A \in \GA$ and $B$ is a $G$-$C^*$-algebra, the assertion follows from the following series of isomorphisms:
    \begin{align*}
        &\operatorname{Map}_{\operatorname{sSet}}(K,\operatorname{Hom}(A,B)_\bullet) \\
        \cong & \operatorname{Map}_{\operatorname{sSet}}(K,\underset{p \in \mathcal{S}(A)}{\operatorname{colim}} \operatorname{Hom}(A_p,B)_\bullet) \\
        \cong & \underset{p \in \mathcal{S}(A)}{\operatorname{colim}} \operatorname{Map}_{\operatorname{sSet}}(K,\operatorname{Hom}(A_p,B)_\bullet) \\
        \cong & \underset{p \in \mathcal{S}(A)}{\operatorname{colim}} \operatorname{Hom}(A_p,C(|K|,B))_\bullet \\
        \cong & \operatorname{Hom}(A,C(|K|,B))_\bullet
    \end{align*}
    Here, for the first and last isomorphisms, we use \cite[Lemma 1.1.5]{Phi}, for the second one, we use that finite simplicial sets are compact objects in $\operatorname{sSet}$ (see \cite[Proposition 3.6.1.9]{kerodon}), and for the third one, we needed the result in the $G$-$C^*$-algebra case as mentioned in the beginning of the proof.

    When both $A$ and $B$ are arbitrary objects of $\GA$, we can consider the commutative diagram
    \[
\begin{tikzcd}
	{\operatorname{Hom}(A,C(|K|,B))_\bullet} && {\operatorname{Map}_{\operatorname{sSet}}(K,\operatorname{Hom}(A,B)_\bullet)} \\
	{\underset{p \in \mathcal{S}(B)}{\operatorname{lim}}\operatorname{Hom}(A,C(|K|,B_p))_\bullet} && {\operatorname{Map}_{\operatorname{sSet}}(K,\operatorname{Hom}(A,\overline{B})_\bullet)}
	\arrow[from=1-1, to=1-3]
	\arrow[hook, from=1-1, to=2-1]
	\arrow[hook, from=1-3, to=2-3]
	\arrow["\cong", from=2-1, to=2-3]
\end{tikzcd}
    ,\]
where the horizontal maps are inclusions of simplicial sets. The bottom map is an isomorphism follows from an ``isomorphism chase" analogous to the one done above, together with the universal property of limits. This already shows that the upper horizontal map is a monomorphism of simplicial sets. The assertion that its an isomorphism now follows by explicitly identifying the simplicial sets upstairs with their images downstairs.
\end{proof}
The proof of the following result is now immediate from the formal properties of limits and colimits, and the fact that every simplicial set is a colimit of its finite simplicial subsets (see \cite[Remark 3.6.1.8]{kerodon}).
\begin{cor}\label{cor: cotensor preserved}
    Let $A,B \in \GA$, and $K$ a simplicial set. Then, we have a natural isomorphism
    \[
    \operatorname{Hom}(A,\C(\lvert K|,B))_\bullet \cong \operatorname{Map}_{\operatorname{sSet}}(K,\operatorname{Hom}(A,B)_\bullet),
    \]
    where \[
    \C(|K|,B)=\underset{L \subseteq K \operatorname{finite}}{\operatorname{lim}}\C(|L|,B).
    \]
    is given the inverse limit topology.
\end{cor}
\begin{rem}
    The topology on $\C(|K|,B)$ in the above result is same as the topology of uniform convergence on compact sets, since every compact subset of $|K|$ is contained in $|L|$ for some finite simplicial subset of $|K|$ (cf. \cite[Corollary 2.9]{Phi2}).
\end{rem}
\subsection{Localization at homotopy equivalences}
We now define the relation of homotopy of continuous $*$-homomorphisms between lmc $G$-C*-algebras.
\begin{defi}\label{defi: homotopy}
    Let $A,B \in \GA$. Then, two $G$-equivariant continuous $*$-homomorphisms $\varphi_0, \varphi_1 : A \rightarrow B$ are said to be homotopic if they are homotopic in the Kan complex $\operatorname{Hom}(A,B)_\bullet$. That is, if there exists a $G$-equivariant continuous $*$-homomorphism $\Phi: A \rightarrow \C([0,1],B)$ such that $\varphi_i=\operatorname{ev}_i \circ \Phi$.
\end{defi}
The following remarks are in order:
\begin{rem}
    \hfill
    \begin{itemize}
        \item \cref{defi: homotopy} is equivalent to asserting that there exists a morphism of simplicial sets $\Delta[1] \rightarrow \operatorname{Hom}(A,B)_\bullet$ such that the restriction at $\{0\}$ and $\{1\}$ yield $\varphi_0$ and $\varphi_1$ respectively; see \cite[Section 3.4]{HOV}.
        \item From the general theory of Kan complexes (see \cite[Section 3.1]{kerodon}), we conclude that homotopy is an equivalence relation, and is well-behaved with respect to compositions. Furthermore, tensor product of homotopic morphisms are also homotopic.
        \end{itemize}
\end{rem}
\begin{defi}
    We denote the $\infty$-category presented by the Kan enriched category $\GA$ as $\GA_h$, i.e., by taking the homotopy coherent nerve of $\GA$ (see \cite[Proposition 1.1.5.10]{HTT})
\end{defi}
\begin{lem}\label{lem: homotopy localization}
    $\GA_h$ is equivalent to the Dwyer-Kan localization of $\GA$ at the collection of homotopy equivalences
\end{lem}
\begin{proof}
    The proof of Proposition 3.5 of \cite{BD} works verbatim in $\GA$, in light of \cref{prop: tensored and cotensored}.
\end{proof}
\begin{cor}\label{cor: homotopy (co)products}
    $\GA_h$ has all products and coproducts, and the localization functor $\GA \rightarrow \GA_h$ preserves them.
\end{cor}
\begin{proof}
    We just discuss the case of products, the case of coproducts follows directly from the definition of a colimit. 
    
    Let $B_i$ be any small family of elements in $\GA$, and $A \in \GA$ be arbitrary. It suffices to show that the map
    \[
    \operatorname{Hom}(A,\prod_i B_i)_\bullet \rightarrow \prod_i\operatorname{Hom}(A,B_i)_\bullet,
    \]
    induced by the family of projections $(\prod_iB_i \rightarrow B_j)_{j \in I}$ is an isomorphism of Kan complexes. But this follows since the natural map $\C(X,\prod_iB_i) \rightarrow \prod_i\C(X,B_i)$ is a homeomorphism for any compact space $X$, see \cite[Section XII.5]{DUG}. Note that we implicitly use (the equivariant version of) \cref{rem: lim top non-eq}. 
\end{proof}
\begin{defi}
    $A \in \GA$ is said to be contractible if $A$ is a zero object in $\GA_h$.
\end{defi}
\begin{lem}\label{lem: zero object path algebra}
    For any $A \in \GA$, $\mathcal{C}_*(I,A)$ is contractible.
\end{lem}
\begin{proof}
    In light of \cref{prop: tensored and cotensored}, we have, for any $A' \in \GA$,
    \[
    \operatorname{Hom}(A',\C_*(I,A))_\bullet \cong \operatorname{Map_{sSet}}_*(\Delta^1,\operatorname{Hom}(A',A)),
    \]
    and the latter anima is (weakly) contractible by \cite[Lemma I.7.5]{GJ}.
\end{proof}
Before proceeding, we note that the tensor product constructed in \cref{cor: ten pro} descends to the homotopy localizations.
\begin{cor}\label{cor: fully faithful oo cat.}
    The localizations $L_h:G\text{-}C^*\text{-alg}\rightarrow G\text{-}C^*\text{-alg}_h$ and $L_h:\GA \rightarrow \GA_h$ admit symmetric monoidal refinements, and we have a fully faithful symmetric monoidal functor $G\text{-}C^*\operatorname{-alg}_h \hookrightarrow \GA_h$.
\end{cor}
\begin{proof}
    In light of \cite[Proposition 3.3.2]{HIN}, the localization $\GA_h$ has a symmetric monoidal refinement since for all $A \in \GA, A \otimes -: \GA \rightarrow \GA$ preserves homotopy equivalences from \cref{cor: ten pro} (one could also use \cref{prop: ten pro simp enrichment} instead). The case of $G\text{-}C^*\operatorname{-alg}_h$ follows similarly, see \cite[Lemma 3.3]{BD} for the details. The rest follows from \cref{rem: cpct open} and \cref{lem: homotopy localization}.
\end{proof}
\subsection{Enforcing stability}
We now discuss how the categories $\GA$ and $\GA_h$ behave under tensoring with the algebra of compact operators on a separable $G$-Hilbert space. We begin by fixing some notation, which will be used throughout the remainder of the article.
\begin{notn}
    Throughout this article, $l^2\N$ denotes the standard $\ell^2$ space, equipped with the trivial $G$-action. $L^2G_\N$ denotes the $G$-Hilbert space $L^2G \otimes l^2\N$ (with the diagonal $G$-action). We also need to work with the following algebras and morphisms in this article:
    \[\kg:=\mathcal{K}(L^2G_\N),~~~~ \hatkg:=\mathcal{K}((L^2G_\N) \oplus l^2\N),~~~ \alpha:\kg \hookrightarrow \hatkg~\text{the natural inclusion.}\]
\end{notn}
Before proceeding, we record a Dixmier-Douady type result about compact operators on a $G$-Hilbert space. In what follows, $\h,\h'$ are $G$-Hilbert spaces, and $\operatorname{Iso}(\h,\h')$ denotes the space of $G$-equivariant isometries from $\h$ to $\h'$, equipped with the strong-$*$ topology. It follows from \cite[Lemma 4.3]{Mey} (cf. \cite[Lemma 8.8]{D}) that $\operatorname{Iso}(\h,l^2\N \otimes \h')$ is weakly contractible\footnote{the argument in \cite[Lemma 4.3]{Mey} shows that the space is path-connected. However, redoing the argument for Hilbert $\C(X,B),G$-modules for every compact Hausdorff space $X$ would yield weak contractibility too.}. As a consequence, each isometry $\iota: \h \rightarrow L^2G_\N$ is homotopic to a unitary, and each isometry $L^2G_\N \rightarrow L^2G_\N$ is homotopic to the identity. As an immediate corollary to this, we obtain:
\begin{cor}\label{cor: dix-doua for kg}
    Any two $G$-equivariant $*$-homomorphisms from  $\mathcal{K}(\h)$ to $\mathcal{K}(L^2G_\N)$ which are induced by $G$-equivariant isometries from $\h$ to $L^2G_\N$ are homotopic.
\end{cor}
The following result justifies why we would like to work with the algebra $\hatkg$ at a technical level.
\begin{cor}\label{cor: hatkg projection}
    There exists a preferred homotopy class of $*$-homomorphism
    \[
    A \xrightarrow{\varepsilon} A \otimes \hatkg
    \]
    for any $A \in \GA$.
\end{cor}
\begin{proof}
    The map can be obtained by choosing an isometry from $\mathbb{C}$ onto the $l^2\N$ summand of $L^2G_\N \oplus l^2\N$, and any two such isometries are homotopic from the discussion above..
\end{proof}
We record here another crucial technical result, which will be very useful in Section 5.
\begin{lem}\label{lem: cpt preserve lim}
   Given an inverse system $\{B_\beta\}_{\beta}$ in $\GA$, the natural map
    \[
    \Phi:(\lim\limits_{\beta} B_\beta) \otimes \mathcal{K}(\h) \xrightarrow{\cong} \lim\limits_{\beta}(B_\beta \otimes \mathcal{K}(\h))
    \]
    is an isomorphism in $\GA$.
\end{lem}
\begin{proof}
    Using the fact that the forgetful functor $\GA \rightarrow \mathbb{A}^\nu$ preserves limits and tensor products (see \cref{prop: eq cat} and \cref{cor: forget tensor}), the statement boils down to proving bijectivity of the map 
    \[
    \mathcal{K}(\iota^*)\Phi\mathcal{K}(\iota):(\lim\limits_{\beta} B_\beta) \otimes \mathcal{K}(l^2) \xrightarrow{\cong} \lim\limits_{\beta}(B_\beta \otimes \mathcal{K}(l^2)),
    \]
    where $\iota: \h \rightarrow l^2$ is a (non-equivariant) unitary. But this isomorphism is the content of Lemma 6.1 of \cite{MJ}.
\end{proof}
We next study the $\infty$-category obtained by inverting the $\kg$-equivalences. 
\begin{defi}
    A map $f: A \rightarrow B$ in $\GA$ is said to be a $\kg$-equivalence if $f \otimes 1_\kg: A \otimes \kg \rightarrow B \otimes \kg$ is a homotopy equivalence. Denote the $\GA_{h,\kg}$ the Dwyer-Kan localization of $\GA$ at homotopy equivalences and the $\kg$-equivalences, and $S_\kg:\GA_h \rightarrow \GA_{h,\kg}$ the localization functor. 
\end{defi}
We omit some details in the proof of the following lemma, as they are discussed in \cite{BUNBEN}. 
\begin{lem}\label{lem: kg loc functor}
\hfill
\begin{enumerate}
    \item The functor $S_\kg$ is the composite of the left Bousfield localization $\GA_h \rightarrow L_{\hatkg}\GA_h$, followed by the right Bousfield localization $L_{\hatkg}\GA_h \rightarrow R_\kg L_{\hatkg}\GA_h$, where $L_{\hatkg}:=-\otimes \hatkg, R_{\kg}:=-\otimes \kg$.
    \item The localization functor
    $S_{\kg}:\GA_{h} \rightarrow 
    \GA_{h,\kg}$ has a symmetric monoidal refinement. 
    \item $\GA_{h,\kg}$ admits finite coproducts and arbitrary small products, and $S_{\kg}$ preserves them.
\end{enumerate}
\end{lem}
\begin{proof}
    Assertion 1 follows directly from the discussion in \cite[Remark 3.34]{BUNBEN}\footnote{see the proof of \cref{lem: uli} for some elaborations.}. Assertion 2 follows from \cite[Proposition 3.3.2]{HIN}, in light of \cref{prop: ten pro simp enrichment}.
    
    For Assertion 3, the case of finite coproducts follows as in \cite[Section 3.4]{BUNBEN}. We briefly recall the idea for the convenience of the reader. Since $\GA_{h,\kg}$ is the Dwyer-Kan localization of $\GA_h$ at $\kg$-equivalences, it inverts, in particular, the left upper-corner inclusions $A \rightarrow A \otimes \mathcal{K}(H)$, where $H$ is any separable, Hilbert space with the trivial $G$-action. This implies, as in the proof of \cite[Lemma 4.5]{BD}, that for any algebras $A,B \in \GA$, the canonical map $A*B \rightarrow A \times B$ is mapped under $S_\kg$ to an equivalence in $\GA_{h,\kg}$. The case of arbitrary small products follows from \cref{lem: cpt preserve lim} and \cref{cor: homotopy (co)products}.
\end{proof}
\begin{rem}\label{rem: map space in skg}
    From \cref{lem: kg loc functor}.1 and general properties of Bousfield localizations, we deduce that 
    \[
    \operatorname{Map}_{L_{\hatkg}\GA_h}(A,B) \simeq \operatorname{Map}_{\GA_h}(A \otimes \hatkg,B \otimes \hatkg) \simeq \operatorname{Map}_{\GA_h}(A,B \otimes \hatkg) 
    \]
    and 
    \[
    \operatorname{Map}_{R_\kg L_{\hatkg}\GA_h}(A,B) \simeq \operatorname{Map}_{\GA_h}(A \otimes \hatkg \otimes \kg,B \otimes \hatkg \otimes \kg).
    \]
\end{rem}
\begin{cor}\label{cor: loc kg semi additive}
    $\GA_{h,\kg}$ is a semi-additive category. In particular, the mapping spaces in $\GA_{h,\kg}$ have canonical refinements to commutative monoids in anima.
\end{cor}
\begin{proof}
     It follows from the series of equivalences
    \[
    S_{h,\kg}(A) \coprod S_{h,\kg}(B) {~\simeq~} S_{h,\kg}(A*B) {~\simeq~} S_{h,\kg}(A \times B) {~\simeq~} S_\kg(L_h(A) \times L_h(B)) {~\simeq~} S_{h,\kg}(A) \prod S_{h,\kg}(B),
    \]
    where the third equivalence follows from \cref{cor: homotopy (co)products}, and the rest follow from (the proof of) \cref{lem: kg loc functor}.3. The last assertion is a general fact about semi-additive categories, see \cite[Remark 6.1.6.14]{HA}.
\end{proof}
\begin{rem}\label{rem: monoid}
    The monoid structure on the mapping spaces in $\GA_{h,\kg}$ is an upgradement of the usual block-sum addition on compact operators; see for instance \cite[Example 4.8]{BD}. The monoid structure can also be constructed independently; see \cite[Section 5]{D} for an approach, which can be made homotopy coherent by observing that the symmetric monoidal category of separable $G$-Hilbert spaces and $G$-equivariant isometries under direct sum can be used to model a linear isometries operad (see \cite[Section 2.7]{Operad}). We can use the latter approach to refine $\operatorname{Map}_{\GA_{h}}(A,B \otimes \hatkg)$ to a commutative monoid in anima too.
\end{rem}
At this juncture, we prove a technical lemma on mapping anima in $\GA_h$, which will be ubiquitous for us for the remainder of this article.
\begin{lem}\label{lem: uli}
    For any $A,B \in \GA$, we have a natural equivalence of commutative monoids
    \[
    \operatorname{Map}_{\GA_h}(A \otimes \kg,B \otimes \kg) \xrightarrow[\simeq]{\alpha_*} \operatorname{Map}_{\GA_h}(A \otimes \kg, B \otimes \hatkg)
    \]
\end{lem}
\begin{proof}
    Before proceeding, we mention how the left and right Bousfield localizations mentioned in \cref{lem: kg loc functor}.1 come by. All the details can be found in \cite[Remark 3.34]{BUNBEN}. The map $\varepsilon:\mathbb{C} \rightarrow \hatkg$ realizes $\hatkg$ as an idempotent algebra in $\GA_h$, whence we have the left Bousfield localization $\GA_h \rightarrow L_{\hatkg} \GA_h$ from \cite[Proposition 4.8.2.4]{HA}. Next, $\kg$ becomes an idempotent co-algebra in $L_{\hatkg}\GA_h$ via the map $\alpha: \kg \rightarrow \hatkg$ (note that $\hatkg$ represents the tensor unit in $L_\kg\GA_h$), since we have an equivalence $\kg \otimes \kg \xrightarrow{\alpha \otimes \operatorname{id}} \hatkg \otimes \kg$ from the discussion above \cref{cor: dix-doua for kg}. Therefore, we have the right Bousfield localization
    $
    L_{\hatkg}\GA_h\rightarrow R_{\kg}L_{\hatkg}\GA_h.
    $
    By the formula for mapping spaces in Bousfield localizations and \cref{rem: map space in skg}, we conclude that the map
    induced by $\alpha$
    \[
    \operatorname{Map}_{\GA_h}(A \otimes \hatkg \otimes \kg,B \otimes \hatkg \otimes \kg) \xrightarrow{\alpha_*} \operatorname{Map}_{\GA_h}(A \otimes \hatkg \otimes \kg,B \otimes \hatkg \otimes \hatkg)
    \]
    is an equivalence of commutative monoids. 

     Now, the commutative diagram
    \[
\begin{tikzcd}
	{A \otimes \mathcal{K}_G} && {B \otimes \mathcal{K}_G} \\
	{A \otimes \widehat{\mathcal{K}_G} \otimes \mathcal{K}_G} && {B \otimes \widehat{\mathcal{K}_G} \otimes \mathcal{K}_G} & {B \otimes \widehat{\mathcal{K}_G} \otimes \widehat{\mathcal{K}_G}} & {B \otimes \widehat{\mathcal{K}_G}}
	\arrow["{-\otimes 1_{\mathcal{K}_G}}", from=1-1, to=1-3]
	\arrow["\simeq", from=1-1, to=2-1]
	\arrow["\simeq", from=1-3, to=2-3]
	\arrow["{-\otimes 1_{\widehat{\mathcal{K}_G} \otimes \mathcal{K}_G}}", from=2-1, to=2-3]
	\arrow["{\alpha_*}", from=2-3, to=2-4]
	\arrow["\simeq"', from=2-5, to=2-4]
\end{tikzcd}\]
    yields a natural equivalence of commutative monoids
    \[
    \operatorname{Map}_{\GA_h}(A \otimes \kg,B \otimes \kg) \xrightarrow{\simeq} \operatorname{Map}_{\GA_h}(A \otimes \kg, B \otimes \hatkg),
    \]
    which can be directly checked to be induced by $\alpha$.
\end{proof}
\section{Universal equivariant algebras}
In this section, we describe a framework to construct universal $G$-$C^*$-algebras\footnote{the name might be a misnomer, and is borrowed from \cite{Phi}. In general, the universal algebras would be a typical $\nu$-complete lmc $G$-$C^*$-algebra, which very often would not even be Frechet!} on a $G$-set of generators and relations. As a corollary to the general construction, we will show that the category $\GA$ can be tensored over $\operatorname{sSet}$, which in turn will be necessary for certain constructions in the later sections.

Recall that, in commutative algebra, obtaining a universal $\Z$-algebra on a set of generators and relations is an easy construction: one simply takes the polynomial ring on the set of generators, and quotients out the ideal generated by the relations.
The situation gets more subtle if one asks the same question for $C^*$-algebras. Since $C^*$-algebra homomorphisms are contractive, a necessary condition to have (implicitly) on the relations is a global bound on the norm of every element. For example, this rules out the existence of a ``free" $C^*$-algebra. The study of universal $C^*$-algebras was first introduced in \cite{Shape}.

One related topic that has been missing from the literature is a systematic study of ``universal equivariant $C^*$-algebras". When the group is discrete, a minor modification of the construction in \cite{Shape} would work; see \cite{BUNBEN}. However, for a topological group, the main obstacle is to enforce continuity of the group action in the sense of \cref{defi: G-action}. To enforce such continuity relations, we follow the treatment of universal algebras as given in \cite{Phi}, which allows a much broader class of relations than \cite{Shape}.

As mentioned in the Introduction, the construction here appears to be new. We also hope the results of this section might be of independent interest.
\begin{defi}\label{defi: g set and g rel}
    Let $S$ be a $G$-set, and $R$ a collection of relations on $S$, which make sense for elements of a $G$-$C^*$-algebra, and includes the relations that the maps 
    \[
            G \ni g \mapsto g.s \in S
    \]
    are continuous, for all $s \in S$. We call $(S,R)$ a set of $G$-generators and relations. A $G$-representation of $(S,R)$ in a $G$-$C^*$-algebra $A$ is a $G$-equivariant function $\rho:S \rightarrow A$ such that the elements $\rho(s)$ for $s \in S$ satisfy the relations $R$ in $A$.
\end{defi}
\begin{defi}\label{defn: weak adm}
    A set $(S,R)$ of $G$-generators and $G$-relations is called $G$-weakly admissible if the following conditions are satisfied:
    \begin{enumerate}
        \item The zero map from $S$ to the zero $C^*$-algebra is a $G$-representation of $(S,R)$.
        \item If $\rho$ is a $G$-representation of $(S,R)$ in a $\nu$-complete lmc $G$-$C^*$-algebra $A$, and $B$ is a $\nu$-complete lmc $G$-$C^*$-subalgebra of $A$ containing $\rho(S)$, then $\rho$ is a $G$-representation of $(S,R)$ in $B$.
        \item If $\rho$ is a $G$-representation of $(S,R)$ in a $\nu$-complete lmc-$G$-$C^*$-algebra $A$, and $\varphi: A \rightarrow B$ is a continuous homomorphism to a $G$-$C^*$-algebra $B$, then $\varphi \circ \rho$ is a $G$-rep of $(S,R)$ in $B$.
        \item If $A$ is a $\nu$-complete lmc $G$-$C^*$-algebra, and $\rho:S \rightarrow A$ is a function such that, for every $p \in \mathcal{S}(A)$, the composition of $\rho$ with the canonical map $A \rightarrow A_p$ is a $G$-rep of $(S,R)$ in $A_p$, then $\rho$ is a $G$-representation of $(S,R)$.
        \item If $\rho_1,\rho_2,...,\rho_n$ are $G$-reps of $(S,R)$ in $G$-$C^*$-algebras $A_1,A_2,...,A_n$, then the map \[
        s \mapsto (\rho_1(s),\rho_2(s),...,\rho_n(s))
        \]
        is a $G$-rep of $(S,R)$ in $A_1 \oplus A_2 \oplus ... \oplus A_n$.
    \end{enumerate}
    \end{defi}
    \begin{rem}
        Point 3. in \cref{defn: weak adm} differs from the corresponding definition in \cite{Phi}. We demand post-composition of a $G$-representation by any continuous homomorphism is again a $G$-representation, whereas in \cite[Definition 1.3.4]{Phi}, it's only required for surjective homomorphisms. Our framework is more related to the one discussed in \cite{LOR}.

        However, since we are working with $\nu$-complete algebras, at a technical level, we cannot work under the surjective homomorphism assumption unlike Phillips. The difficulty arises from the fact that it's apriori not clear that the canonical map $A \rightarrow A_p$ is surjective in our case. However, all the relations we care about are indeed allowed in our framework, so we don't lose much.
    \end{rem}
    We are now in a position to state and prove the main result of this section:
\begin{thm}\label{thm: uni alg}
    Let $(S,R)$ be a $G$-weakly admissible set of $G$-generators and relations. Then, there exists a $\nu$-complete lmc-$G$-C*-algebra $C_{G,\nu}^*(S,R)$, equipped with a $G$-representation $\rho:S \rightarrow C_{G,\nu}^*(S,R)$, such that, for any $G$-rep $\sigma$ of $(S,R)$ in a $\nu$-complete lmc-$G$-C*-algebra $B$, there exists a unique $G$-equivariant homomorphism $\varphi: C_{G,\nu}^*(S,R) \rightarrow B$ making the following diagram commute:
    \[
\begin{tikzcd}
	{(S,R)} && B \\
	{C_{G,\nu}^*(S,R)}
	\arrow["\sigma", from=1-1, to=1-3]
	\arrow["\rho"', from=1-1, to=2-1]
	\arrow["{\exists!\varphi}"', dashed, from=2-1, to=1-3]
\end{tikzcd}\]
Moreover, equivariant homomorphisms of this universal algebra are in 1-1 correspondence with the weakly admissible relations.
\end{thm}
\begin{proof}
    Let $F(S)$ denote the free $*$-algebra on the set $S$; note that it carries naturally an action of $G$. Let $D$ be the set of all $G$-invariant $C^*$-seminorms on $F(S)$ of the form 
    \[
    q(x)=\lVert \sigma(x) \rVert ~\forall~x \in F(S)
    \]
    for some $G$-rep $\sigma$ of $S$ in a $G$-C*-algebra. Then take $C_{G,\nu}^*(S,R)$ to be the Hausdorff $\nu$-completion of $F(S)$ in the family of $G$-invariant $C^*$-seminorms $D$, with the obvious map $\rho:S \rightarrow C_{G,\nu}^*(S,R)$. 
    We need to prove the following three statements in bullet points to prove our result:
    \begin{itemize}
        \item \textit{$C_{G,\nu}^*(S,R)$ is a $\nu$-complete lmc.-$G$-C*-algebra.}
    \end{itemize}
    \begin{proof} In view of \cref{prop: eq cat}, the only assertion which needs an argument is that $F(S)$ is a $G$-algebra. That $G$ acts by automorphisms on $C^*_{G,\nu}(S,R)$ is clear since the same is true for $F(S)$. In light of \cref{prop: eq cat}.4, it remains to show that $\forall~p \in F(S)$, the map $g \mapsto \varphi_g(p)$ is continuous. But since the addition, (scalar) multiplication, and adjoint operations on $F(S)$ are continuous, it suffices to verify that $g \mapsto \varphi_g(s)$ is continuous for all $s \in S$, which is evidently true.
    \end{proof}
    \begin{itemize}
        \item \textit{$\rho$ is a $G$-rep.} 
    \end{itemize}
    \begin{proof}
        To see that $\rho$ is a $G$-representation, take $q \in \mathcal{S}(C_{G,\nu}^*(S,R))$. Then, the elementary theory of locally convex spaces yields that $q \leq c\operatorname{max}(q_1,q_2,...,q_n)$ for some $c>0$, and $q_1,q_2,...,q_n \in D$. In fact, since we are dealing with $C^*$-seminorms, $c$ can be taken to be less than or equal to $1$. Indeed, for each $x \in C_{G,\nu}^*(S,R)$, write
        \begin{equation} \label{eqn: 1}
            q(x)=c_x\operatorname{max}(q_1(x),q_2(x),...,q_n(x)).
        \end{equation} It suffices to show that $c_x \leq 1 ~\forall~ x \in C_{G,\nu}^*(S,R)$. Squaring both sides of \cref{eqn: 1} yields
        \begin{align*}
            &q(x^*x)=c_x^2\operatorname{max}(q_1(x^*x),q_2(x^*x),...,q_n(x^*x))\\
            \implies & c_{x^*x}=c_x^2.
        \end{align*}
        If $c_x>1$, we can now get a contradiction by repeated squaring of \cref{eqn: 1}; on the one hand, the constant blows up to infinity, but on the other hand, it is bounded by $c$. 

        Hence, we have $q \leq \operatorname{max}(q_1,q_2,...,q_n)$ for some $q_1,q_2,...,q_n \in D$. By (5) of \cref{defn: weak adm}, there exists a $G$-representation $\sigma$ of $(S,R)$ in a $G$-$C^*$-algebra $B$ such that \begin{equation}\label{eqn: 2}
            q(x) \leq \lVert \sigma(x) \rVert ~\forall~ x \in F(S).
        \end{equation} By (2) of \cref{defn: weak adm}, we may assume that $B$ is generated by $\sigma(S)$. 

        We claim that there is a unique $*$-homomorphism $\varphi:B \rightarrow C^*_{G,\nu}(S,R)_q$ making the following diagram commute
        \[
        \begin{tikzcd}
	{(S,R)} && B \\
	{C_G^*(S,R)} && {C_G^*(S,R)_q = \overline{F(S)/{\ker q}}}
	\arrow["\sigma", from=1-1, to=1-3]
	\arrow["\rho", from=1-1, to=2-1]
	\arrow["\varphi", from=1-3, to=2-3]
	\arrow["{\kappa_q}", from=2-1, to=2-3]
\end{tikzcd}.
        \]
        Indeed, the assignment $\varphi(\sigma(x)) \mapsto \kappa_q\circ \rho(x),~x \in S$ is well-defined and contractive (by \cref{eqn: 2}), whence extends to $B$.
        
        Thus, $\kappa_q \circ \rho$ is a $G$-representation by (3) of \cref{defn: weak adm}, whence $\rho$ is a $G$-representation by (4) of \cref{defn: weak adm}.
    \end{proof}
    \begin{itemize}
        \item $C_{G,\nu}^*(S,R)$ satisfies the desired universal property.
    \end{itemize}
    \begin{proof}
        We first note that the uniqueness assertion is clear from the definition of $C^*_{G,\nu}(S,R)$, since it is a completion of the free $*$-algebra on the set $S$. 
        
        To show existence, we first note that if $B$ is a $G$-$C^*$-algebra, then we are done by the very definition of $C^*_{G,\nu}(S,R)$. For a general Hausdorff $\nu$-complete lmc. $G$-C*-algebra, we consider $B \subseteq \varprojlim\limits_{p \in \mathcal{S}(B)} B_p$, and we are done using (2), (3), (4) of \cref{defn: weak adm}.
    \end{proof}
    The assertion about the $1\text{-}1$ correspondence between the set of equivariant homomorphisms of this universal algebra and the set of weakly admissible relations follows immediately from the universal property of $C^*_{G,\nu}(S,R)$ and point 3. of \cref{defn: weak adm}.
\end{proof}
\begin{example}\label{example}
    \begin{enumerate}
        \item Every $A \in \GA$ can be expressed as a universal $G$-C*-algebra on a $G$-set of generators and relations. Indeed, we can take $S=A$ with the (discrete) action of $G$, and as relations take all $*$-algebraic relations in $A$, together with the relation that the identification map $A \rightarrow S$ is continuous (and the relations $g \mapsto g.s$ are continuous $\forall~s \in S$). Compare \cite[Remark 1.3.7]{Phi}.
        \item We can also obtain a model for coproducts in $\GA$ using a $G$-generators and relations description. Indeed, given a family $\{A_i\}$ in $\GA$, we take $S=\coprod_{i \in I} A_i$ with the obvious group action, and the relations $R$ are all $*$-algebraic relations in each of the individual $A_i$'s, together with the relations that the natural inclusions $A_i \rightarrow S$ are continuous. Compare \cite[Proposition 1.5.3(2)]{Phi}.
    \end{enumerate}
\end{example}
\subsection{Tensoring over simplicial sets}
We now discuss an important construction using these universal $G$-equivariant algebras. 

In this section $\operatorname{Cpt}$ and $\operatorname{Cpt}_*$ denote the categories of compact Hausdorff and pointed compact Hausdorff spaces, respectively.
\begin{const}
    Suppose $A \in \GA$, and $X \in \operatorname{Cpt}$. For $x \in X$, let $A_x$ denote a copy of $A$. We can obtain a model for the free product $*_{x \in X}A_x$ as in 2. of \cref{example}. We put a further relation on the free product asserting that the natural map $A \times X \rightarrow *_{x \in X}A_x$ is continuous. Denote the resulting universal $\nu$-complete lmc. $G$-C*-algebra as $A\otimes X$.

    If $(X,x_0) \in \operatorname{Cpt}_*$, then we can modify the previous construction by requiring the following map to be continuous instead:
    \begin{align*}
        A \times X &\longrightarrow *_{x \in X}A_x \\
        (a,x) &\mapsto \begin{cases}
            a \in A_x~\text{if}~x \neq x_0\\
            0 \in A_{x_0}~\text{if}~x=x_0
        \end{cases}
    \end{align*}
     We denote the resulting $\nu$-complete lmc. $G$-C*-algebra as $A \owedge X$.
\end{const}
\begin{prop}\label{prop: adj unpointed}
    Let $A,B \in \GA$ and $X \in \operatorname{Cpt}$. Then, there is a natural isomorphism of simplicial sets
    \[
        \operatorname{Hom}(A \otimes X,B)_\bullet \rightarrow \operatorname{Hom}(A,\C(X,B))_\bullet
    \]
    natural in $A,B,X$.
\end{prop}
\begin{proof}
    Note that there is a natural map 
    \[
    i : A \times X \rightarrow A \otimes X
    \] by definition, which yields a map
    \[
    i^*:\operatorname{Hom}(A \otimes X, B) \rightarrow \operatorname{Map}_G(A \times X,B).
    \]
    Postcomposing $i^*$ with the map
    \[
    \iota_A: \operatorname{Map}(A \times X, B) \rightarrow 
    \operatorname{Map}(A,\C(X,B)).\]
    restricts to give the map
    \[
        \operatorname{Hom}(A \otimes X,B) \rightarrow \operatorname{Hom}(A,\C(X,B))
    \]
    in the statement. The result now follows directly from the universal properties of the universal algebra; cf. \cite[Proposition 5.3]{MJ}.
\end{proof}
We can deduce the analogous result in the pointed case too, upon noting that the zero homomorphism maps to the zero homomorphism under the natural map mentioned in \cref{prop: adj unpointed}.
\begin{prop} \label{prop: adj pointed}
    Let $A,B \in \GA$ and $(X,x_0) \in \operatorname{Cpt_*}$. Then, there is a natural isomorphism of pointed simplicial sets
    \[
        \operatorname{Hom}(A \owedge X,B)_\bullet \rightarrow \operatorname{Hom}(A,\C_*(X,B))_\bullet
    \]
    natural in $A,B,X$.
\end{prop}
The following remarks are in order:
\begin{rem}\label{rem: main rem}
\begin{enumerate}
    \hfill
    \item 
    The constructions $A \otimes X$ and $A \owedge X$ are clearly functorial in both variables. In particular, we have natural maps $A \rightarrow A \otimes X$, and $A \rightarrow A \owedge X$ for a (pointed) compact space $X$, each obtained by the choice of a point in $X$.
    
    \item In analogy with our experience from spaces, one might wonder whether the two natural maps mentioned above are topological inclusions.

    The map $A \rightarrow A\otimes X$ is indeed a topological inclusion. To see this, one can simply note that there is also a morphism $A \otimes X \rightarrow A$, induced by the map $X \rightarrow *$, which realizes $A$ as a retract of $A \otimes X$.

    The map $A \rightarrow A \owedge X$ however need not always be a topological inclusion. To see this, consider $A=\mathbb{C}$, $G=\{e\}$, $X=[0,1]$, with the basepoint $\{0\}$. We claim that $\mathbb{C} \owedge I$ is the zero algebra. Indeed, by \cref{prop: adj pointed}, we get 
    \begin{equation}\label{eqn: adj pointed}
         \operatorname{Hom}(\mathbb{C} \owedge I,B) \rightarrow \operatorname{Hom}(\mathbb{C},\mathcal{C}_*(I,B))~\forall~B \in \mathbb{A}^{\nu}.
    \end{equation}
    But, since $*$-homomorphisms out of the complex numbers classify projections, and the algebras $\C_*(I,B)$ have no projections (see II.5.5.9 in \cite{BLA}), it follows that the right-hand side of \cref{eqn: adj pointed} is zero for all $B \in \mathbb{A}^{\nu}$. The conclusion follows.

    In \cref{sec: sep}, we use the so-called stable cone seminorm extension property to circumvent this problem.

    \item
    More generally, for any $B \in \GA$, $B \owedge I$ is a zero object in $\GA$. Indeed, for any $A \in \GA$, we have
    \[
    \operatorname{Hom}(B \owedge I,A)_\bullet \cong \operatorname{Hom}(B,\C_*(I,A))_\bullet \simeq 0,
    \]
    where the last equivalence follows from \cref{lem: zero object path algebra}.
\end{enumerate}
\end{rem}
As an immediate consequence of our work so far, we have the following important corollaries, which show that $\GA$ is tensored over $\operatorname{sSet}$.
\begin{cor}\label{cor: adjunction}\label{prop: partial adjunction unpointed}
    For $K \in \operatorname{sSet}$, upon defining 
    \[
    A \otimes |K|:=\underset{L \subseteq K\operatorname{finite}}{\operatorname{colim}} A \otimes |L|,
    \]
    the functor
    \[
    A \otimes |-|: \operatorname{sSet} \rightarrow \GA
    \]
    is an enriched left adjoint to the functor
    \[
    \operatorname{Hom}(A,-)_\bullet:\GA \rightarrow \operatorname{sSet}
    \]
\end{cor}
\section{Adjoint of tensoring with compacts on a $G$-Hilbert space}
The goal of this section is to come up with a left adjoint of the functor
\[
- \otimes \mathcal{K}(\h): \GA \rightarrow \GA,
\]
where $\h$ is a separable $G$-Hilbert space. As mentioned in the Introduction, this will be ubiquitous for us to verify the axioms of a model category on $\GA$. We would obtain the existence of the left adjoint via the Freyd's adjoint functor theorem:

\begin{thm}[{\cite[Theorem V.6.2]{Mac}},{\cite[Theorem X.1.2]{Mac}}]\label{thm: aft}
Given a small-complete category $A$ with small hom-sets, a functor $G:A \rightarrow X$ has a left adjoint if and only if it preserves all small limits and satisfies the so-called \textit{solution set condition:}

For each object $x \in X$ there is a small set $I$ and an $I$-indexed family of arrows $f_i:x \rightarrow Ga_i$ such that every arrow $h:x \rightarrow Ga$ can be written as a composite $h=Gt \circ f_i$ for some index $i$ and some $t:a_i \rightarrow a$.  

Moreover, when this is the case, a left adjoint $F$ is given on each $x \in X$ as 
\begin{equation}\label{eqn: lim}
Fx=\varprojlim(Q:(x \downarrow G) \rightarrow A),
\end{equation}
where $x \downarrow G$ denotes the comma category over $G$, and $Q$ denotes the projection onto the second factor. The left adjunct of each arrow $g:x \rightarrow Ga$ is the component $\lambda_g:Fx \rightarrow Qg=a$ of the limiting cone $\lambda$ for the limit \cref{eqn: lim}.
\end{thm}
\begin{rem}
    Note that the limit in \cref{eqn: lim} can in fact be taken to be a small limit in light of the solution set condition; see \cite{inifun}. Thus, we have 
    \[
    GFx=\varprojlim(G(Q:(x \downarrow G)\rightarrow A)),
    \]
    and the unit of the adjunction is the canonical morphism 
    \[
    x \xrightarrow{\varprojlim_{i} f_i} \varprojlim_{i} Ga_i.
    \]
\end{rem}
In light of \cref{lem: cpt preserve lim}, we only have to show that the functor
\[
- \otimes \mathcal{K}(H): \GA \rightarrow \GA.
\]
satisfies the solution set condition. Given any $X \in \GA$ and a morphism $\varphi: X \rightarrow A \otimes \mathcal{K}(H)$ in $\GA$, $\varphi$ factors through a $\nu$-complete $G$-lmc-$C^*$-algebra generated by a set of cardinalily $\leq |X|$. So, the required small set maybe taken to isomorphism classes of all $\nu$-complete $G$-lmc $C^*$-algebras on a generating set of cardinality $\leq |X|$. Thus, by \cref{thm: aft}, we get:
\begin{prop}\label{prop: adj}
    There exists a functor $- \boxtimes \mathcal{K}(\h) : \GA \rightarrow \GA$ \footnote{$- \boxtimes \mathcal{K}(\h)$ is just some notation, should not be confused with any version of tensor product} for which we have a natural bijection:
    \begin{equation}\label{eqn: adj}
        \operatorname{Hom}(A \boxtimes \mathcal{K}(\h),B) \rightarrow \operatorname{Hom}(A, B \otimes \mathcal{K}(\h))
    \end{equation}
    for all $A,B \in \GA$.
\end{prop}
In light of \cref{prop: CXB}, we can upgrade the adjunction in \cref{prop: adj} to an enriched one, as stated in the following result. This yields a key desired property, which was not possible to attain in the framework discussed in \cite{MJ}.
\begin{cor}\label{prop: hom class}\label{cor: kk_g eq}
    We have a natural isomorphism of simplicial sets 
    \[
    \operatorname{Hom}_{G}(A \boxtimes \mathcal{K}(\h),B)_\bullet \rightarrow \operatorname{Hom}_{G}(A, B \otimes \mathcal{K}(\h))_\bullet.
    \]
    In particular, for a morphism $B \rightarrow B'$ in $\GA$, the induced map
    \[
    \operatorname{Hom}_{G}(A \boxtimes \widehat{\mathcal{K}_G},B)_\bullet \rightarrow \operatorname{Hom}(A \boxtimes \widehat{\mathcal{K}_G},B')_\bullet
    \] is a (weak) homotopy equivalence iff 
    the adjoint
    \[
    \operatorname{Hom}(A, B \otimes \widehat{\mathcal{K}_G})_\bullet \rightarrow \operatorname{Hom}(A, B' \otimes \widehat{\mathcal{K}_G})_\bullet 
    \]
    is a (weak) homotopy equivalence. 
\end{cor}
We end this section by drawing a few important corollaries to the above result.
\begin{cor}\label{cor: inf cat adjunction}
    We have adjunctions
    \[
    \begin{tikzcd}
\GA_h
\arrow[r, "- \boxtimes \mathcal{K}(\h)"{name=F}, bend left=25] &
\GA_h
\arrow[l, "- \otimes \mathcal{K}(\h)"{name=G}, bend left=25]
\arrow[phantom, from=F, to=G, "\dashv" rotate=-90]
\end{tikzcd}.
    \]
    In particular, we have adjunctions \[
    \begin{tikzcd}
\operatorname{ho}(\GA_h)
\arrow[r, "- \boxtimes \mathcal{K}(\h)"{name=F}, bend left=25] &
\operatorname{ho}(\GA_h)
\arrow[l, "- \otimes \mathcal{K}(\h)"{name=G}, bend left=25]
\arrow[phantom, from=F, to=G, "\dashv" rotate=-90]
\end{tikzcd}.
    \]
\end{cor}
\begin{proof}
    In light of \cref{prop: hom class}, the first assertion follows directly from \cref{lem: homotopy localization} and \cite[Corollary 5.2.4.5]{HTT}. The second assertion follows from \cite[Proposition 5.2.2.9]{HTT}.
\end{proof}
    In light of \cref{cor: hatkg projection} and \cref{cor: inf cat adjunction}, we conclude that there is a preferred homotopy class of $*$-homomorphism
    \[
    A \boxtimes \hatkg \xrightarrow{\varepsilon^{\operatorname{ad}}} A
    \]
    for each $A$ in $\GA$. At this juncture, we mention the following technical result, which is going to crucial for us in the later sections.
\begin{cor}\label{prop: B boxtimes K mapsto B is weak eq}
    The functor $- \boxtimes \hatkg:\GA_h \rightarrow \GA_h$ is fully faithful, i.e., $\GA_h \rightarrow L_{\hatkg}\GA_h$ is a right Bousfield localization. Thus, for all $A,B \in \GA$, we have a natural equivalence
     \[
    \operatorname{Map}_{\GA_h}(A \boxtimes \hatkg, B \boxtimes \hatkg) \xrightarrow{\varepsilon^{\operatorname{ad}}_*} \operatorname{Map}_{\GA_h}(A \boxtimes \hatkg,B).
    \]
\end{cor}
\begin{proof}
    The first assertion follows immediately from \cite[Corollary 2.7]{RAMZI}, in light of \cref{lem: kg loc functor}.1. The second assertion follows from the general properties of right Bousfield localizations. 
\end{proof}
\begin{rem}
    A more direct proof of \cref{prop: B boxtimes K mapsto B is weak eq} can be deduced similar to \cite[Proposition 6.8]{MJ}.
\end{rem}
\section{$KK$-theory with $\nu$-complete lmc $G$-$C^*$-algebras}
In this section, we discuss the neccessary background on equivariant $KK$-theory. In particular, we come up with a modification of the Cuntz-Meyer picture of $KK^G$ (cf. \cref{thm: Intro}) which will be relevant for our purposes. 
\subsection{The $\mathfrak{q}$-construction and equivariant $KK$-theory}
We begin with a study of Cuntz's $q$-construction. Though the results are well known in the $C^*$-algebra case (see \cite[Chapter 5]{ELEMENTS}, \cite{CUN}), we also need to use some basic properties of it for general elements in $\GA$, so we include some details.
\begin{defi}
    For any $A \in \GA$, set $QA:=A *A \in \GA$. Denote by $\iota,\overline{\iota}$ the two inclusions of $A$ into $QA$. We set $qA$ to be the closed, two-sided, ideal of $QA$ generated by the set $\{\iota(x)-\overline{\iota}(x) \mid x \in A\}$.
\end{defi}
The following remarks are in order:
\begin{rem}\label{rem: q}
    \hfill
    \begin{enumerate}
        \item Note that $qA$ is automatically $G$-invariant, as each of the generators of $qA$ are.
        \item Given two morphisms $\alpha_1,\alpha_2:A \rightarrow B \in \GA$, we denote by $Q(\alpha_1,\alpha_2):QA\rightarrow B$ the induced map out of $A*A$. We denote by $q(\alpha_1,\alpha_2)$ the restriction of $Q(\alpha_1,\alpha_2)$ to $qA$.
        \item There exists a linear map $q:A \rightarrow qA$, mapping $a$ to $\iota(a)-\overline{\iota}(a)$. A direct calculation shows that $q(xy)+qxqy=\iota x qy+qx\iota y$.
    \end{enumerate}
\end{rem}
The following lemma is a direct generalization\footnote{at one point, we need to work with nets instead of sequences.} of \cite[Lemma 5.1.2]{ELEMENTS}, and the remarks following it, to general topological algebras.
\begin{lem}\label{lem: q const}
    For any $A \in \GA$, we have $qA=\ker Q(\operatorname{id}_A,\operatorname{id}_A)$. The $q$-construction is functorial. Moreover, the map $q(\operatorname{id}_A,0):qA \rightarrow A$ is natural in $G$-equivariant $*$-homomorphisms $A \rightarrow B$.
\end{lem}
\begin{defi}
    Let $A,B \in \GA$. A prequasi-homomorphism from $A$ to $B$ consists of a $\nu$-complete lmc $G$-$C^*$-algebra $E$, two $G$-equivariant $*$-homomorphisms $\alpha,\overline{\alpha}:A \rightarrow E$, an ideal $J \subseteq E$, and a $G$-equivariant $*$-homomorphism $\mu:J\rightarrow B$ such that $\alpha(x)-\overline{\alpha}(x) \in J~\forall~x\in A$. We write it as
    \[
\begin{tikzcd}
	{A } & {E \triangleright J} & B
	\arrow["{\overline{\alpha}}"', from=1-1, to=1-2]
	\arrow["\alpha", shift left=3, from=1-1, to=1-2]
	\arrow["\mu", from=1-2, to=1-3]
\end{tikzcd}
    \]
\end{defi}
    By universal properties, a prequasihomomorphism $A \rightarrow B$ yields a map $\mu \circ q(\alpha,\overline{\alpha}):qA \rightarrow B$.
\begin{cor}
    For any $A,B \in \GA$, there exists a natural map $\kappa:q(A \otimes B) \rightarrow A \otimes qB$.
\end{cor}
\begin{proof}
    Consider the map $q(A \otimes B) \rightarrow A \otimes qB$ induced by the prequasi-homomorphism 
    \[
\begin{tikzcd}
	{A \otimes B} & {A \otimes QB \triangleright A \otimes qB} & {A \otimes qB}
	\arrow["{1 \otimes \overline{\iota}}"', from=1-1, to=1-2]
	\arrow["{1 \otimes \iota}", shift left=3, from=1-1, to=1-2]
	\arrow["{\operatorname{id}}", from=1-2, to=1-3]
\end{tikzcd}.
    \]
\end{proof}
We now discuss an equivariant version of the Cuntz construction, which will be relevant for us.
\begin{defi}
    Let $A$ in $\GA$. We set
    \[
    \mathfrak{q}A:=q(A \otimes \kg) \otimes \kg.
    \]
\end{defi}
Clearly, the $\mathfrak{q}$-construction is also functorial. From \cref{lem: q const}, we have a map \[\mathfrak{q^2}A=q(q(A \otimes \kg)\otimes \kg \otimes \kg)\otimes \kg \rightarrow q(A \otimes \kg)\otimes \kg \otimes \kg \otimes \kg.\]
    In light of \cref{cor: dix-doua for kg}, this yields a map, unique and functorial upto homotopy,
    \[
    \pi_{\mathfrak{q}A}:q(q(A \otimes \kg)\otimes \kg)\otimes \kg \rightarrow q(A \otimes \kg) \otimes \kg.
    \]
    We will succinctly denote this map as $\pi_{\mathfrak{q}A}:\mathfrak{q}^2A\rightarrow\mathfrak{q}A$ in this article.
    
    We now present a brief discussion about classical equivariant $KK$-theory using the $\mathfrak{q}$-construction. The connection first appeared in the generality of locally compact groups in \cite{Mey}, and some generalizations were recently made in \cite{CG}.
\begin{defi}[\cite{CG},~Definition 6.1]\label{thm: Mey}
    Let $KK^{G}_{\operatorname{class}}$ denote the category whose objects are separable $G$-$C^*$-algebras, and morphisms are given by \[
    KK^{G}_{\operatorname{class}}(A,B)=[\mathfrak{q}A, B \otimes \kg],
    \]
    and the morphisms are given by the composition in \cref{eqn: product}.
\end{defi}
\begin{lem}\label{lem: class kk}
    $KK^{G}_{\operatorname{class}}$ is equivalent to each of the following categories
    \begin{enumerate}
        \item Kasparov's classical equivariant $KK$-theory category (\cite{KAS}).
        \item The category obtained by localizing the category of separable $G$-$C^*$-algebras at $KK^G$-equivalences.
    \end{enumerate}
\end{lem}
\begin{proof}
    These follows from the results in \cite[Propositions 2.1,2.2]{BEL} and \cite[Remark 6.2(b)]{CG}, see also \cite[Theorem 5.5, Theorem 6.6]{Mey}.
\end{proof}
\begin{rem}\label{rem: cg kk functor}
    We have a natural functor from the category of separable $G$-$C^*$-algebras to $KK^{G}_{\operatorname{class}}$, mapping $\varphi:A \rightarrow B$ to the homotopy class of the composite
    \[
    \mathfrak{q}A \xrightarrow{\pi_A} A \otimes \kg \xrightarrow{\varphi \otimes 1_\kg} B \otimes \kg.
    \]
    Note that, it is same (upto homotopy) to the other potential composite
    \[
    \mathfrak{q}A \xrightarrow{q(\varphi \otimes 1_\kg) \otimes \kg} q(B \otimes \kg) \otimes \kg \xrightarrow{\pi_B} B \otimes \kg.
    \]
    This natural functor is the universal functor from separable $G$-$C^*$-algebras to any category which inverts $KK^G$-equivalences; see the proof of \cite[Proposition 2.1]{BEL}.
\end{rem}
In light of \cref{thm: Mey}, we work with the following definitions in the remainder of this article.
\begin{defi}
    Let $A$ be a separable $G$-$C^*$-algebra, and $B \in \GA$. We define
    \[
    KK^G(A,B)_\bullet:= \operatorname{Hom}(\mathfrak{q}A, B \otimes \kg)_\bullet,
    \]
    which naturally refines to a commutative monoid in anima by \cref{rem: map space in skg} and \cref{cor: loc kg semi additive}\footnote{see also \cref{rem: monoid}.}.
    A morphism $B \rightarrow B'$ in $\GA$ is said to be a left $KK^G$-equivalence if the induced map $KK^G(A,B)_\bullet \rightarrow KK^G(A,B')_\bullet$ is a (weak) homotopy equivalence for all separable $G$-$C^*$-algebras $A$. $B \in \GA$ is said to be left $KK^G$-contractible if $KK^G(A,B)_\bullet$ is (weakly) contractible for all separable $G$-$C^*$-algebras $A$.
\end{defi}
\begin{lem}
    Let $A$ be a separable $G$-$C^*$-algebra, and $B \in \GA$. Then, $KK^G(A,B)_\bullet$ naturally refines to a commutative group in anima.
\end{lem}
\begin{proof}
    By \cref{lem: class kk} and classical Kasparov theory, we conclude that $\pi_0KK^G(A,B)_\bullet$ is a group. Since a commutative monoid $M$ in anima is a group if and only if $\pi_0(M)$ is a group (see \cite[Example 5.2.6.4]{HA}), the conclusion follows.
\end{proof}
 The following result is an immediate corollary of \cref{lem: uli}.
\begin{cor}\label{cor: datta picture}
    For a separable $G\text{-}C^*$-algebra $A$, and $B \in \GA$,  we have a natural equivalence of commutative groups
    \[
    KK^G(A,B)_\bullet \simeq \operatorname{Hom}(\mathfrak{q}A, B \otimes \widehat{\mathcal{K}_G})_\bullet.
    \]
\end{cor}
\begin{cor}\label{cor: kk equi}
    Let $A$ be a separable $G$-C*-algebra, and $B$, $B'$ be $\sigma$-unital $G$-C*-algebras. Then the two equivalent assertions of \cref{cor: kk_g eq} are equivalent to the induced map
    \[
    KK^G_{*}(A,B) \rightarrow KK^G_{*}(A,B')
    \]
    being an isomorphism.
\end{cor}
\begin{proof}
    This essentially follows from the series of isomorphisms
    \begin{align*}
        \pi_n\operatorname{Hom}(\mathfrak{q}A,B \otimes \widehat{\mathcal{K}_G})_\bullet &\cong \pi_0\operatorname{Hom}(\mathfrak{q}A,\mathcal{C}_*(S^n,B \otimes \hatkg))_\bullet \\
        &\cong \pi_0\operatorname{Hom}(\mathfrak{q}A,\C_*(S^n,B)\otimes \hatkg)_\bullet\\
        &\cong KK^G_0(A,\C_*(S^n,B)) \\
        &\cong KK^G_n(A,B),
    \end{align*}
    all of which are standard; see \cite{ECH} for example. Note that, we didn't need to specify basepoints before taking the homotopy groups since a map of grouplike $H$-spaces is a weak equivalence if and only if it is a $\pi_*$ isomorphism at the basepoint $0$.
\end{proof}
\subsection{The Kasparov product in the Cuntz-Meyer picture}
We next discuss the construction of the Kasparov product using \cref{thm: Mey}. This has been discussed in \cite{CG}, but we take a slightly different approach which would be helpful for our applications. We don't give all the details; they can be deduced from the ideas in \cite{CG}; see also \cite{BD},\cite{CUN}.
The key to constructing Kasparov product in this framework is the following.
\begin{thm}\label{thm: the map}
    For any separable $G$-$C^*$-algebra $A$, there exists a map, well-defined and functorial upto homotopy
    \[
    \varphi_A: \mathfrak{q}A \rightarrow \mathfrak{q}^2A.
    \]
    Furthermore, $\varphi_A$ is a homotopy equivalence with homotopy inverse given by $\pi_{\mathfrak{q}A}$.
\end{thm}
\cref{thm: the map} follows from the following result, in light of \cref{cor: dix-doua for kg}.
\begin{prop}
    For any separable $G$-$C^*$-algebra $A$, there exists a map, well-defined and functorial upto homotopy
    \[
    \varphi_A: \mathfrak{q}A \rightarrow M_2(\mathfrak{q}^2A).
    \]
    Furthermore, $\varphi_A$ is a stable homotopy equivalence, in the sense that $M_2(\pi_{\mathfrak{q}A}) \circ \varphi_A: \mathfrak{q}A \rightarrow M_2(\mathfrak{q}A)$ and $\varphi_A \circ \pi_{\mathfrak{q}A}: \mathfrak{q}^2A \rightarrow M_2(\mathfrak{q}^2A)$ are homotopic to the left upper corner inclusions.
\end{prop}
\begin{proof}
    We begin by setting up some notation. In order to not confuse between the different occurences of $\kg$, we consider three algebras of compact operators on $G$-Hilbert spaces, denoted $\kg_1,\kg_2$. Consider the maps $\epsilon,\overline{\epsilon},\eta,\overline{\eta}$ as prescribed by the following diagram
    \[
\begin{tikzcd}
	{A \otimes \kg_1 \otimes \kg_2 {}} & {Q(A \otimes \kg_1)\otimes \kg_2 {}} & {Q(Q(A \otimes \kg_1)\otimes \kg_2) {}}
	\arrow["{\overline{\epsilon}}"', from=1-1, to=1-2]
	\arrow["\epsilon", shift left=3, from=1-1, to=1-2]
	\arrow["\eta", shift left=3, from=1-2, to=1-3]
	\arrow["{\overline{\eta}}"', shift right=3, from=1-2, to=1-3]
\end{tikzcd}
    \]
    Set
    \[
    R=C^*\{\begin{pmatrix}
        R_1&R_1R_2 \\
        R_2R_1 & R_2
    \end{pmatrix}\} \subseteq Q(q(A \otimes \kg_1)\otimes \kg_2){},
    \]
    where $R_1:=\eta(q(A \otimes \kg_1)\otimes \kg_2 {}), R_2:=\overline{\eta}(q(A \otimes \kg_1)\otimes \kg_2 {})$.
    
Define \[
D:=C^*\{\begin{pmatrix}
    \eta\epsilon(A \otimes \kg_1 \otimes \kg_2) & 0 \\
    0 & \overline{\eta}{\epsilon}(A \otimes \kg_1 \otimes \kg_2 {})
\end{pmatrix}\} \subseteq M_2({Q(Q(A \otimes \kg_1)\otimes \kg_2){}}).
\]
Note that, we have a canonical surjection
\[
Q(q(A \otimes \kg_1)\otimes \kg_2) {} \twoheadrightarrow q(A \otimes \kg_1)\otimes \kg_2 {}
\]
with kernel $q(q(A \otimes \kg_1)\otimes \kg_2) {} $. This yields, in light of the Cohen factorization theorem (see \cite{COHEN}), that there exists a surjection
\[
p:R \twoheadrightarrow M_2(q(A \otimes \kg_1) \otimes \kg_2 {}).
\]
Set
\[
J:=\ker p \supseteq M_2(q(q(A \otimes \kg_1)\otimes \kg_2){}).
\]
Thus, we obtain 
\[
\frac{R}{J} \cong M_2(q(A \otimes \kg_1) \otimes \kg_2 {}),~~~~~\frac{R+D}{J} \cong M_2(q(A \otimes \kg_1) \otimes \kg_2 {}+C^*\{\begin{pmatrix}
    \epsilon(A \otimes \kg_1 \otimes \kg_2 {}) & 0 \\
    0 & \epsilon(A \otimes \kg_1 \otimes \kg_2 {})
\end{pmatrix}\}.
\]
Using \cite[Proposition 2.2]{CG}, we can lift the multiplier $S_0=\begin{pmatrix}
    0 & 1 \\1 & 0
\end{pmatrix}$ of $\frac{R+D}{J}$ to a self-adjoint multiplier $S$ of $R+D$ that commutes mod $J$ with $D$, and which satisfies $\alpha_g(S)-S \in J ~\forall~g \in G$.

This multiplier $S$ can be naturally extended to a $G$-invariant multiplier $S'$ of $\mathcal{K}(L^2G) \otimes R$, which furthermore commutes mod $\mathcal{K}(L^2G) \otimes J$ with $\mathcal{K}(L^2G) \otimes D$, as follows from \cite[Section 6]{CG}. Set
\[
F':=e^{\frac{\pi i}{2}S'}, \sigma':=\operatorname{Ad}F': \mathcal{M}(\mathcal{K}(L^2G) \otimes J) \rightarrow \mathcal{M}(\mathcal{K}(L^2G) \otimes J).
\]
Tensoring by $\mathcal{K}(L^2G)$, we extend the maps $\eta\epsilon, \eta\overline{\epsilon},\overline{\eta}\epsilon,\overline{\eta}\overline{\epsilon}$ to homomorphisms
\[
\mathcal{K}(L^2G) \otimes A \otimes \kg_1 \otimes \kg_2 {} \rightarrow \mathcal{K}(L^2G) \otimes Q(Q(A \otimes \kg_1)\otimes \kg_2) {}.
\]
Then, the pair of homomorphisms
\[
\begin{pmatrix}
    \begin{pmatrix}
        \eta\epsilon & 0\\
        0 & \overline{\eta\epsilon}
    \end{pmatrix}, \sigma'\begin{pmatrix}
        \overline{\eta}\epsilon & 0\\
        0 & \eta\overline{\epsilon}
    \end{pmatrix}
\end{pmatrix}
\]
defines an equivariant homomorphism 
\[
\varphi_0:q(\mathcal{K}(L^2G) \otimes A \otimes \kg_1 \otimes \kg_2 {}) \rightarrow \mathcal{K}(L^2G) \otimes J \hookrightarrow \mathcal{K}(L^2G) \otimes M_2(q(q(A \otimes \kg_1)\otimes \kg_2) {}),
\]
which in turn yields
\[
1_\kg \otimes \varphi_0: \kg \otimes q(\mathcal{K}(L^2G) \otimes A \otimes \kg_1 \otimes \kg_2  {}) \rightarrow \kg \otimes \mathcal{K}(L^2G) \otimes M_2(q(q(A \otimes \kg_1)\otimes \kg_2) {}).
\]
We can now consider $\kg_1=\kg_2=\kg$, and obtain, in light of \cref{cor: dix-doua for kg}, a map
\[
\varphi_A:\mathfrak{q}A \rightarrow M_2(\mathfrak{q}^2A).
\]
We now just mention the homotopies as sought for in the statement of the Proposition. The details can be deduced as in the proof of \cite[Theorem 1.6]{CUN},\cite[Proposition 3.4]{CG}. For each $t\in [0,1]$, we set
\[
F_t':=e^{\frac{\pi it}{2}S'}, \sigma_t':=\operatorname{Ad}F_t': \mathcal{M}(\mathcal{K}(L^2G) \otimes J) \rightarrow \mathcal{M}(\mathcal{K}(L^2G) \otimes J), \tau_t=\begin{pmatrix}
    \operatorname{cos}(\frac{\pi t}{2}) & -\sin(\frac{\pi t}{2})\\
    \operatorname{sin}(\frac{\pi t}{2}) & \cos(\frac{\pi t}{2})
\end{pmatrix} 
\]
A homotopy $\gamma_t$ from the left upper corner inclusion $\mathfrak{q}^2A \rightarrow M_2(\mathfrak{q}^2A)$ to $\varphi_A\circ \pi_{\mathfrak{q}A}$ has the components $1_\kg \otimes \widehat{\gamma^i_t}:q(A \otimes \kg) \otimes \kg \rightarrow M_2(Q(q(A \otimes \kg)\otimes \kg))\otimes \kg \cong M_2(Q(q(A \otimes \kg)\otimes \kg)\otimes \kg)$ given by the map
\[
\widehat{\gamma^0_t}:=q\begin{pmatrix}
    \begin{pmatrix}
        \eta\epsilon& 0\\
        0 & \overline{\eta}\overline{\epsilon}
    \end{pmatrix},\sigma'_t\begin{pmatrix}
        \eta\overline{\epsilon} & 0 \\
        0 & \overline{\eta}\epsilon
    \end{pmatrix}
\end{pmatrix}\otimes 1_\kg, \widehat{\gamma^1_t}:=q\begin{pmatrix}
    \begin{pmatrix}
        \overline{\eta}\epsilon & 0\\
        0 & \overline{\eta}\overline{\epsilon}
    \end{pmatrix},\operatorname{Ad}(\tau_t)\begin{pmatrix}
        \overline{\eta}\overline{\epsilon} & 0\\
        0 & \overline{\eta}\epsilon
    \end{pmatrix}
\end{pmatrix} \otimes 1_\kg.
\]
A homotopy $\lambda_t$ from the left upper corner inclusion $\mathfrak{q}A \rightarrow M_2(\mathfrak{q}A)$ to $M_2(\pi_{\mathfrak{q}A})\circ \varphi_A$ is defined as the composite
\[
\mathfrak{q}A \xrightarrow{\widehat{\gamma^0_t} \otimes 1_\kg} M_2(Q(q(A \otimes \kg)\otimes \kg)\otimes \kg)\xrightarrow{M_2(p\otimes 1_\kg)} M_2(q(A \otimes \kg)\otimes \kg\otimes \kg) \cong M_2(\mathfrak{q}A),
\]
where $p$ denotes the canonical projection $Q(\operatorname{id}_{q(A \otimes \kg)\otimes \kg},0)$.
\end{proof}
The map $\varphi_A$ of \cref{thm: the map} induces the associative Kasparov product $KK^G(A,B) \times KK^G(B,C) \rightarrow KK^G(A,C)$ as follows: let elements $\mu \in KK^G(A,B),\nu \in KK^G(B,C)$ be given by 
\[
\mu: \kg \otimes q(\kg \otimes A) \xrightarrow{}\kg \otimes B,~~~~~~~~~ \nu:\kg \otimes q(\kg \otimes B) \xrightarrow{}\kg \otimes C.
\] Consider the composite given by
\begin{equation}\label{eqn: product}
   \kg \otimes q(\kg \otimes A) \xrightarrow{\varphi_A} \kg\otimes q(\kg \otimes q(\kg \otimes A)) \xrightarrow{1_\kg \otimes q(\mu)} \kg \otimes q(\kg \otimes B) \xrightarrow{\nu} \kg \otimes C, 
\end{equation}
which represents the product of $\nu\#\mu \in KK^G(A,C)$.
We now mention a simplification of \cref{eqn: product} when one of $\mu$ and $\nu$ come from an honest $G$-equivariant $*$-homomorphism; see also \cite[Remark 3.3]{CG}.
\begin{lem} \label{lem: composition}
    Let $A,B,C$ be separable $G$-$C^*$-algebras.
    If $\mu:A\rightarrow B$ be a $G$-equivariant $*$-homomorphism, then $\nu \#\mu$ is given by $\nu \circ (1_\kg \otimes q(1_\kg \otimes \mu))$. Similarly, if $\nu: B \rightarrow C$ be a $G$-equivariant $*$-homomorphism, then $\nu \# \mu$ is given by $(\nu \otimes 1_\kg)\circ \mu$.
\end{lem}
\begin{cor}\label{cor: composition of inverses}
    Let $B,C$ be separable $G$-$C^*$-algebras. Suppose $f:B \rightarrow C$ is a $G$-equivariant $*$-homomorphism, which is a $KK^G$-equivalence, with inverse given by $g:\kg \otimes q(\kg \otimes C) \rightarrow B \otimes \kg$.
    Then, the composite
    \[
    \kg \otimes q(B \otimes \kg) \xrightarrow{1_\kg \otimes q(f \otimes 1_\kg)} \kg \otimes q(C \otimes \kg) \xrightarrow{g}B \otimes \kg
    \]
    is homotopic to the morphism
    \[\kg \otimes q(B \otimes \kg) \xrightarrow{1_\kg \otimes \pi_{B \otimes \kg}} \kg \otimes B \otimes \kg \xrightarrow{\cong}B \otimes \kg\]
\end{cor}
\begin{proof}
    The latter morphism is, by definition, the image of $1_B \in \operatorname{Hom}(B,B)$ inside $KK^G(B,B)$, which becomes the identity inside $KK^G(B,B)$ in light of \cref{rem: cg kk functor}. By \cref{lem: composition}, the former morphism is $g\# f$. The rest follows since $g=f^{-1}$ in $KK^G(C,B)$.
\end{proof}
\subsection{Some left $KK^G$-equivalences}
We now discuss some algebras in $\GA$ which are left $KK^G$-equivalent. Apart from providing some natural examples, the results of this subsection will also be crucially used while proving stability of our model structure in Section 8.

Before proceeding, we remind the readers that there is a category of fibrant objects structure on separable $G\text{-}C^*\!\operatorname{-alg}$ whose weak equivalences are the $KK^G_{}$-equivalences, see \cite[Proposition 2.10]{BEL}. In particular, from general properties of category of fibrant objects, we get that  every $KK^G$ equivalence between separable $G$-$C^*$-algebras is a zig-zag of $*$-homomorphisms which are $KK^G$-equivalences.  

The key technical result of this section is the following.
\begin{lem}\label{lem: technical stability}
    Let $f:B \rightarrow C$ be a $*$-homomomorphism between separable $G$-$C^*$-algebras which is a $KK^G$-equivalence. Then, the $*$-homomorphism $A \otimes B \xrightarrow{1_A \otimes f} A \otimes C$ is a left $KK^G$-equivalence for every $A \in \GA$.
\end{lem}
\begin{proof}
    Let \[g:\mathfrak{q}C \rightarrow B \otimes \kg\] be an inverse to $f$.
    We have to show that, for any separable $G$-$C^*$-algebra $D$, the map of Kan complexes
    \[
    1_A\otimes f \otimes 1_{\kg} : \operatorname{Hom}(\mathfrak{q}D,A \otimes B \otimes \kg)_\bullet \rightarrow \operatorname{Hom}(\mathfrak{q}D,A \otimes C \otimes \kg)_\bullet
    \]
    is a homotopy equivalence. To that end, we come up with a map of Kan complexes 
    \[
    \mathfrak{g}: \operatorname{Hom}(\mathfrak{q}D,A \otimes C \otimes \kg)_\bullet \rightarrow \operatorname{Hom}(\mathfrak{q}D,A \otimes B \otimes \kg)_\bullet
    \]
    such the composites $\Phi:=\mathfrak{g}\circ(1_A\otimes f \otimes 1_{\kg})$ and $\Psi:=(1_A\otimes f \otimes 1_{\kg})\circ \mathfrak{g}$ induce the respective identity maps on $\pi_0$; this is sufficient as we can replace $A$ by $\C_*(S^n,A)$ to get an isomorphism of higher homotopy groups.

    We begin by the definition of $\mathfrak{g}$. On the $0$-simplices, we define it as follows: given $\nu \in \operatorname{Hom}(\mathfrak{q}D,A \otimes C \otimes \kg)$, we set $\mathfrak{g}(\nu)$ to be the composite
    \[
    \mathfrak{q}D\xrightarrow{\varphi_D}  \kg \otimes q(\kg \otimes q(D \otimes \kg)) \xrightarrow{1_\kg \otimes q(\nu)} \kg \otimes q(A \otimes C \otimes \kg) \rightarrow A \otimes \kg \otimes q(\kg \otimes C) \xrightarrow{1_A \otimes g} A \otimes B \otimes \kg.
    \]
    The definition of the higher simplices can be easily deduced from this. Given $\mu \in \operatorname{Hom}(\mathfrak{q}D,A\otimes B \otimes \kg)$, from \cref{eqn: product}, we can write $\Phi(\mu)=\mathfrak{g} \circ (1_A \otimes f \otimes 1_\kg)(\mu)$ as the composite
    \begin{equation} \label{eqn: Varphi}
\begin{tikzcd}
	{\mathfrak{q}D} & {\mathcal{K}_G \otimes q(\mathcal{K}_G \otimes q(\mathcal{K}_G \otimes D))} & {\mathcal{K}_G \otimes q(A \otimes B \otimes \mathcal{K}_G)} \\
	& {} & {\mathcal{K}_G \otimes q(A \otimes C \otimes \mathcal{K}_G)} \\
	&& {\mathcal{K}_G \otimes A \otimes q(\mathcal{K}_G \otimes C)} \\
	&& {A \otimes B \otimes \mathcal{K}_G}
	\arrow["{\varphi_D}", from=1-1, to=1-2]
	\arrow["{1_{\mathcal{K}_G \otimes q(\mu)}}", from=1-2, to=1-3]
	\arrow["{1_{\mathcal{K}_G}\otimes q(1_A \otimes f \otimes 1_{\mathcal{K}_G})}", from=1-3, to=2-3]
	\arrow[from=2-3, to=3-3]
	\arrow[from=3-3, to=4-3]
\end{tikzcd}
    \end{equation}
    Now, note that the vertical part of the composition in \cref{eqn: Varphi} is homotopic to the map
    \[
    \kg \otimes q(A \otimes B \otimes \kg) \xrightarrow{1_\kg \otimes \kappa} \kg \otimes A \otimes q(B \otimes \kg) \xrightarrow{1_{\kg \otimes A} \otimes \pi_{B \otimes \kg}} \kg \otimes A \otimes B  \otimes \kg \xrightarrow{\cong} A \otimes B \otimes \kg,
    \]
    as can be seen from \cref{cor: composition of inverses} and \cref{prop: ten pro simp enrichment}.

    Thus, to show that $\Phi$ induces identity on $\pi_0$, it suffices to show that the following diagram commutes upto homotopy 
    \[
\begin{tikzcd}
	{\mathcal{K}_G \otimes q(D \otimes \mathcal{K}_G)} && {A \otimes B \otimes \mathcal{K}_G \cong A \otimes B \otimes \mathcal{K}_G \otimes \mathcal{K}_G} \\
	{} && {A \otimes \mathcal{K}_G \otimes q(B \otimes \mathcal{K}_G)} \\
	{\mathcal{K}_G \otimes q(\mathcal{K}_G\otimes q(\mathcal{K}_G\otimes D))} && {\mathcal{K}_G\otimes q(A \otimes B \otimes \mathcal{K}_G)}
	\arrow["\mu", from=1-1, to=1-3]
	\arrow["{\varphi_D}", from=1-1, to=3-1]
	\arrow[from=2-3, to=1-3]
	\arrow["{1_{\mathcal{K}_G} \otimes q\mu}", from=3-1, to=3-3]
	\arrow[from=3-3, to=2-3]
\end{tikzcd}
    \]
    The left vertical arrow is a homotopy equivalence by \cref{thm: the map} with homotopy inverse $\pi_{\mathfrak{q}A}$, and the rest follows from the naturality of the $q$-construction.
    
    The case of $\Psi$ is analogous.
\end{proof}
As an immediate corollary, we have:
\begin{prop}
    Suppose $B$ and $C$ are $KK^G$-equivalent separable $G$-$C^*$-algebras. Then, for any $A \in \GA$, $A \otimes B$ and $A \otimes C$ are weakly equivalent in $\GA$.
\end{prop}

\begin{cor}\label{cor: Bott}
    $\C_0(\R^2,A)$ and $A$ are left $KK^G$-equivalent.
\end{cor}
\begin{proof}
    Note that $\C_0(\R^2,A) \cong \C_0(\R^2) \otimes A$ from \cref{prop: CXB}, and the rest follows from the classical $KK$-theory fact that $\C_0(\R^2)$ and $\mathbb{C}$ are $KK^G$-equivalent; see \cite[Section 3.4]{ECH}, for example.
\end{proof}
We end this section by a folklore result, which will be needed in the later sections.
\begin{prop}\label{cor: stable cone seminorm on B(H)}
    Every $G$-$C^*$-algebra $B$ embeds into a left $KK^G$-contractible $G$-$C^*$-algebra. 
\end{prop}
\begin{proof}
    From the general theory of covariant representations, we get that $B$ admits a faithful covariant representation on $\mathcal{H}_G:=L^2G \otimes \mathcal{H}$ for any (non-equivariant) infinite dimensional Hilbert space $\mathcal{H}$ on which $B$ admits a faithful representation (see \cite[Lemma C.1.9]{HIT}). This induces an injective morphism of $G$-$C^*$-algebras
    \[
    B \rightarrow \mathcal{B}(\mathcal{H}_G)^c,
    \]
    where the superscript $c$ stands for the closed $*$-subalgebra of $\mathcal{B}(\mathcal{H}_G)$ where the group action is continuous.

    We will now show that $\mathcal{B}(\mathcal{H}_G)^c$ is left $KK^G$-contractible. Let $A$ be a separable $G$-$C^*$-algebra. We first show that $KK^G_0(A,\mathcal{B}(\mathcal{H}_G)^c)=0$. It suffices to show that for any $[\varphi] \in KK^G_0(A,\mathcal{B}(\mathcal{H}_G)^c)$, there exists a $[\psi] \in KK^G_0(A,\mathcal{B}(\mathcal{H}_G)^c)$ such that $[\varphi]+[\psi]=[\psi]$. We first note that there are two possible additions on 
    \[
    KK^G_0(A,\mathcal{B}(\mathcal{H}_G)^c)=[q(A \otimes \kg)\otimes \kg,  \mathcal{B}(\mathcal{H}_G)^c\otimes \mathcal{K}(L^2G \otimes l^2\N)],
    \]
    coming from the identifications $\mathcal{H}_G \cong \mathcal{H}_G \oplus \mathcal{H}_G$, and $L^2G \otimes l^2\N \cong L^2G \otimes l^2\N \oplus L^2G \otimes l^2\N$. However, they are the same by a direct Eckmann-Hilton argument. We use the former identification, which allows us to form a countably infinite sum of $\varphi$, by using the identifications $\mathcal{H}_G \cong \oplus_{i=1}^{\infty}\mathcal{H}_G$. Indeed, we have
    \[
    \oplus_{i=1}^{\infty}\varphi: q(A \otimes \kg)\otimes \kg \rightarrow  \mathcal{B}(\oplus_{i=1}^{\infty}\mathcal{H}_G)^c\otimes \mathcal{K}(L^2G \otimes l^2\N)],
    \]
    is a well defined $G$-equivariant bounded linear operator, and the isomorphism $\mathcal{H}_G \cong \oplus_{i=1}^{\infty}\mathcal{H}_G$ yields an isomorphism $\mathcal{B}(\oplus_{i=1}^{\infty}\mathcal{H}_G)^c \cong \mathcal{B}(\mathcal{H}_G)^c$, inducing an element $KK^G_0(A,\mathcal{B}(\mathcal{H}_G))$ which we denote by $[\oplus_{i=1}^{\infty}\varphi]$. As any two $G$-equivariant unitaries between infinite dimensional $G$-Hilbert spaces are path-connected in the norm topology \footnote{this follows by noting that the set of $G$-equivariant unitaries on a $G$-Hilbert space $\h$ is the unitary group of the commutant of $G$ in $\mathcal{B}(H)$, which is a von-Neumann algebra.}, we conclude that 
    \[
    [\varphi] + [\oplus_{i=1}^{\infty}\varphi] = [\oplus_{i=1}^{\infty}\varphi]~\text{in}~KK^G_0(A,\mathcal{B}(\mathcal{H}_G)^c),
    \]
    as required.

    The case of higher $KK$-groups now follows from the above proof, by noting that $\C_*(X,\mathcal{B}(\mathcal{H}_G)^c) \cong \C_*(X) \otimes \mathcal{B}(\mathcal{H}_G)^c$ for any (pointed) compact Hausdorff space $X$. 
\end{proof}
\begin{rem}
    A proof of \cref{cor: stable cone seminorm on B(H)} may also be constructed by formal properties of the functor $KK^G(A,-)$ functor, similar to proof of \cite[Proposition 2.7.7]{HIT}.
\end{rem}
\section{The seminorm extension property}\label{sec: sep}
In this section, we study a class of morphisms in $\GA$ which will include all the cell-complex inclusions in our sought after cofibrantly generated model structure. These are the morphisms having the seminorm extension property; see \cref{defi: sep}. Our goal is to show that every object in $\GA$ is \textit{small} with respect to such morphisms \footnote{and (a version of) this is needed to come up with the cofibrantly generated model structures using the Quillen's small object argument, see \cite[Section 2.1]{HOV}.}; the verification of this gets quite technical since objects in a category of locally convex algebras are manifestly \textit{cosmall}. This is where working with $\nu$-complete algebras helps us, since, albeit $\nu$-completion being a transfinite procedure, a ``large enough" sequential colimit of $\nu$-complete algebras is automatically $\nu$-complete; see \cref{lem: technical 3}. 

We mention that the analogous treatment in \cite{MJ} is inaccurate since the description of the $\nu$-completion functor therein was incorrect.
\begin{defi}\label{defi: sep}
    A morphism $f:A \rightarrow B$ in $G\operatorname{-lmc}C^*\!\operatorname{-alg}$ is said to have the seminorm extension property if the induced map $f^*:\mathcal{S}(B) \rightarrow \mathcal{S}(A)$ is surjective. That is, if every $G$-invariant continuous $C^*$-seminorm on $A$ can be extended along $f$ to a $G$-invariant continuous $C^*$-seminorm on $B$.
\end{defi}
We next present a few examples to demonstrate some morphisms which have the seminorm extension property.
\begin{example}\label{example 2} \hfill
\begin{enumerate}
    \item The canonical maps $A \rightarrow A^{\operatorname{Haus}}$ (where $G\operatorname{-lmc}C^*\!\operatorname{-alg}$) and $A \rightarrow \overline{A}^\nu$ (where $A \in \GA$) clearly have the seminorm extension property. As the seminorm extension property is closed under composition, we conclude that the Hausdorff $\nu$-completion of lmc-$G$-$C^*$-algebra (\cref{rem: Hausdorff nu completion}) also has the seminorm extension property.
    \item If $X$ is a retract of $Y$, then for any $A \in \GA$, the inclusion $A \otimes X \hookrightarrow A \otimes Y$ has the seminorm extension property, as is immediate from the functoriality of the construction.
\end{enumerate}
\end{example}
While dealing with Hausdorff algebras, the seminorm extension property is a upgradation of topological inclusion, as the following lemma suggests. We omit the proof, since it is a direct equivariant analogue of \cite[Lemma 7.3]{MJ}.
\begin{lem} \label{lem: sep implies inclusion}
A homomorphism $f: A \to B$ in $G\text{-}\mathbb{A}$ that has the seminorm extension property is a topological inclusion. 
\end{lem}
Our next goal would be to show that objects of $\GA$ are small with respect to morphisms having the seminorm extension property. As mentioned at the beginning of this section, the analogous statement in the non-equivariant case is \cite[Lemma 7.5]{MJ}, but the proof therein is wrong. So, we redo all the proofs here. 

We begin by recalling the definition of a small object in a category, as there are many competing notions with the same name in the literature. Our definitions and conventions in this regard follow \cite[Section 2.1]{HOV}. We refer the reader to \cite[Definition 2.1.1, Definition 2.1.2]{HOV} for the definition of $\lambda$-sequences and $\gamma$-filtered ordinals.
\begin{defi}
    Suppose $\mathcal{C}$ is a category with all small colimits, $\mathcal{D}$ a collection of morphisms in $\C$, $A$ is an object of $\C$, and $\kappa$ a cardinal. We say that $A$ is \textit{$\kappa$-small relative to $D$} if, for all $\kappa$-filtered ordinals $\lambda$ and all $\lambda$-sequences 
    \[
    X_0 \rightarrow X_1 \rightarrow ...\rightarrow X_\beta \rightarrow...
    \]
    such that each map $X_\beta \rightarrow X_{\beta+1}$ is in $\mathcal{D}$ for $\beta+1 < \lambda$, the map of sets
    \[
    \operatorname{colim}_{\beta < \lambda} \C(A,X_\beta) \rightarrow \C(A,\operatorname{colim}_{\beta <\lambda} X_\beta)
    \]
    is an isomorphism.
\end{defi}
We begin with a few technical lemmas, with the eventual goal to establish that every object in $\GA$ is small with respect to morphisms having the seminorm extension property.
\begin{lem}\label{lem: forget pres filtered colim}
    The forgetful functor
    $ G\text{-}\!\operatorname{lmc}\!C^*\operatorname{-alg}\rightarrow \operatorname{Sets}
    $
    preserves filtered colimits.
\end{lem}
\begin{proof}
    Since the forgetful functor $G\text{-}\!\operatorname{lmc}\!C^*\operatorname{-alg} \rightarrow \operatorname{lmc}\!C^*\operatorname{-alg}$ preserves colimits (\cref{cor: forget preserve lim/colim}), it suffices to show that the forgetful functor 
    $\operatorname{lmc}\!C^*\operatorname{-alg} \rightarrow \operatorname{Sets}
    $
    preserves filtered colimits. To that end, we consider the model of colimits in $\operatorname{lmc}\!C^*\operatorname{-alg}$ as described in \cref{rem: colimit seminorm}, and invoke the well known algebraic fact that the forgetful functor from $\operatorname{*-alg}\rightarrow \operatorname{Sets}$ preserves filtered colimits (see \cite[Proposition IX.2]{Mac}, and the discussion that follows it).
\end{proof}
\begin{lem} \label{lem: technical 2}
    Let $F: \lambda \rightarrow G\text{-}\mathbb{A}$ be a diagram in $G\text{-}\mathbb{A}$, where $\lambda$ is an ordinal. Suppose that all the maps $A_\alpha \rightarrow A_{\alpha+1}$ have the seminorm extension property, and the induced map 
    \[
    \underset{\alpha < \beta}{\operatorname{colim}}^{G\operatorname{-lmc-C*-alg}} A_\alpha\rightarrow A_\beta
    \]
    too have the seminorm extension property for every limit ordinal $\beta < \lambda$\footnote{$F$ is not necessarily a $\lambda$-sequence in this lemma!}. Then, the underlying set of the colimit of $F$ is isomorphic to the colimit of the underlying sets of the $A_\alpha$'s. Furthermore, the bonding maps of the universal cocone
    \[
    f_\alpha : A_\alpha \rightarrow  \underset{\alpha < \lambda}{\operatorname{colim}}^{G\text{-}\mathbb{A}} A_\alpha = \operatorname{colim}_{G\text{-}\mathbb{A}}F
    \]
    too have the seminorm extension property.
\end{lem}
\begin{proof}
    From \cref{cor: eq bicom}, we know that 
    \[
    \operatorname{colim}_{G\text{-}\mathbb{A}}F \cong (\operatorname{colim}_{G\text{-}\!\operatorname{lmc}\!C^*\operatorname{-alg}}F)^{\operatorname{Haus}}.
    \]
    In light of \cref{lem: forget pres filtered colim}, it suffices to show that \[
    A:=\operatorname{colim}_{G\text{-}\!\operatorname{lmc}\!C^*\operatorname{-alg}}F
    \] is already Hausdorff. 

    We claim that any two distinct points in $A$ are separated by a seminorm in $\mathcal{S}(A)$. To see this, let $x,y \in A$ be two distinct elements. Since $\lambda$ is a filtered category, an explicit description of filtered colimits in $\operatorname{Sets}$ (see the paragraph following \cite[Theorem 3.8.9]{EMILY}, for example), we can, without loss of generality, assume that $x=f_\alpha(x')$ and $y=f_\alpha(y')$, where $\alpha < \lambda$. As $A_\alpha$ is Hausdorff, there is a $p' \in \mathcal{S}(A_\alpha)$ such that $p'(x-y) \neq 0$.
    
    We now show using transfinite induction that there exists a $p \in \mathcal{S}(A)$ such that $p \circ f_\alpha = p'$. The successor ordinal case is handled from the very definition of the seminorm extension property. For each limit ordinal $\beta$, we see, from \cref{lem: forget pres filtered colim}, that we can extend $p$ to a well defined continuous $G$-equivariant $C^*$-seminorm \[
    p'': \underset{\alpha < \beta}{\operatorname{colim}}^{G\operatorname{-lmc-C*-alg}}A_\alpha \rightarrow \mathbb{R}
    \]
    This can further be extended to $p''':A_\beta \rightarrow \R$ by the given hypothesis. This concludes the proof of the first part of the lemma.

    The final assertion of the lemma also follows similarly, by transfinitely extending each seminorm on $A_\alpha$ to $\underset{\alpha < \lambda}{\operatorname{colim}}^{G\text{-}\mathbb{A}} A_\alpha$.
\end{proof}
\begin{lem}\label{lem: technical 3}
    Let $F: \lambda \rightarrow \GA$ be a $\lambda$-sequence\footnote{this lemma could also be stated in the generality of \cref{lem: technical 2}, but we don't have any need of doing so.}, where $\lambda$ is a $\nu$-filtered cardinal. Suppose that all the maps $A_\alpha \rightarrow A_{\alpha+1}$ have the seminorm extension property. Then,
    \[
    \underset{\alpha < \lambda}{\operatorname{colim }}^{\operatorname{G\text{-}\mathbb{A}}} A_\alpha
    \]
    is $\nu$-complete.
\end{lem}
\begin{proof}
    Let $(a_\bullet)$ be a $\nu$-Cauchy sequence in $
    \underset{\alpha < \lambda}{\operatorname{colim }}^{\operatorname{G\text{-}\mathbb{A}}} A_\alpha
    $. From \cref{lem: forget pres filtered colim} and \cref{lem: technical 2}, we can choose, for each $\nu_0 \in \nu$, an $\alpha_{\nu_0} \in \lambda$ such that $a_{\nu_0} \in f_{\alpha_{\nu_0}}(A_{\alpha_{\nu_0}})$. Then, 
    \[
    \operatorname{sup}\{\alpha_{\nu_0} \mid \nu_0 \in \nu \} < \lambda,
    \]
    as $\lambda$ is $\nu$-filtered. So, there exists an $\alpha_0$ such that \[
    (a_\bullet) \subseteq f_{\alpha_0}(A_{\alpha_0}) \subseteq \underset{\alpha < \lambda}{\operatorname{colim }}^{\operatorname{G\text{-}\mathbb{A}}} A_\alpha.
    \]
    By the last assertion of \cref{lem: technical 2} and \cref{lem: sep implies inclusion}, we conclude 
    \[
    A_{\alpha_0} \cong_{G\text{-}\mathbb{A}} f_{\alpha_0}(A_{\alpha_0}).
    \]
    Thus, as $A_{\alpha_0}$ is $\nu$-complete by assumption, we conclude that $(a_\bullet)$ converges in $\underset{\alpha < \lambda}{\operatorname{colim }}^{\operatorname{G\text{-}\mathbb{A}}} A_\alpha$.
\end{proof}
We can now prove the main result of this section.
\begin{prop}\label{prop: small wrt sep}
    Let $B \in \GA$. Set $\kappa=\max(\lvert B \rvert,\nu)$. Then, $B$ is $\kappa$-small relative to morphisms having the seminorm extension property.
\end{prop}
\begin{proof}
    Let $\lambda$ be a $\kappa$-filtered ordinal, and let $F: \lambda \rightarrow \GA$ be a $\lambda$-sequence, where all the maps $A_\alpha \rightarrow A_{\alpha+1}$ have the seminorm extension property. Let
    \[
    \varphi: B \rightarrow \underset{\alpha < \lambda}{\operatorname{colim }}^{\operatorname{G\text{-}\mathbb{A}^\nu}} A_\alpha = \underset{\alpha < \lambda}{\operatorname{colim }}^{\operatorname{G\text{-}\mathbb{A}}} A_\alpha
    \]
    be any $G$-equivariant continuous *-homomorphism (the last equality follows from \cref{lem: technical 3} and \cref{cor: eq bicom}). Now, 
 note that the composite
 \[
    \lambda \xrightarrow{F} \GA \hookrightarrow G\text{-}\mathbb{A} 
 \] is not necessarily a $\lambda$-sequence anymore, however, it still meets the hypothesis of \cref{lem: technical 2}, in light of \cref{example 2}.1. Thus, we conclude that $\varphi$ induces a set map
    \[
    \varphi: B \rightarrow \underset{\alpha < \lambda}{\operatorname{colim }}^{\operatorname{Sets}} A_\alpha.
    \]
    By \cite[Example 2.1.5]{HOV}, we conclude $\varphi$ factors through
    \[
\begin{tikzcd}
	B && {\underset{\alpha < \lambda}{\operatorname{colim }}^{\operatorname{Sets}} A_\alpha} \\
	& {A_{\alpha_0}}
	\arrow["\varphi", from=1-1, to=1-3]
	\arrow["{\varphi_0}"', from=1-1, to=2-2]
	\arrow[hook, from=2-2, to=1-3]
\end{tikzcd}
    \]
    for some $\alpha_0 < \lambda$. Apriori, $\varphi_0$ is just a set map. But, we can show that it is in fact an equivariant continuous $*$-homomorphism using the fact that all the bonding morphims of universal cocone are (topological) inclusions (see also the proof of \cite[Lemma 2.4.1]{HOV}). This shows that the canonical map
    \[
    \underset{\alpha < \lambda}{\operatorname{colim}}\operatorname{Hom}_{\GA}(B,A_\alpha) \rightarrow
    \operatorname{Hom}_{\GA}(B, \underset{\alpha < \lambda}{\operatorname{colim }}^{\operatorname{G\text{-}\mathbb{A}^\nu}} A_\alpha)
    \]
    is surjective. Invoking again the fact that all the bonding morphisms of the universal cocone are injective will also yield that the above map is injective. This completes the proof.
\end{proof}
We now state a couple of lemmas demonstrating some closure properties of maps having the seminorm extension property. These are direct equivariant analogues of \cite[Lemma 7.7, Lemma 7.8]{MJ}, so we omit the proofs\footnote{At one point in the proof of \cite[Lemma 7.8]{MJ}, the authors stated without justification that a certain pushout of an isometry of $C^*$-algebras is an isometry. This isn't completely formal, and follows from \cite[Theorem 4.2]{PED}, for example. The equivariance doesn't pose any additional difficulty since colimits in $\GA$ are computed in $\mathbb{A}^\nu$ (\cref{cor: forget preserve lim/colim}).}. 
\begin{lem}\label{lem: JJ07}
    Suppose $f:A\rightarrow B, g:B \rightarrow C$ are morphisms in $\GA$ such that $gf$ has the seminorm extension property. Then, $f$ too has the seminorm extension property.
\end{lem}
\begin{lem}\label{lem: JJ08}
    Let
    \[
    B \xleftarrow {f}A \xrightarrow{g} C
    \]
    be a diagram in $\GA$, with $f: A \rightarrow B$ having the seminorm extension property. Then, the induced map $h:C \rightarrow B*_A C$ also has the seminorm extension property.
\end{lem}
We end this section by one technical property circling the seminorm extension property, which will come handy in the next section.
\begin{defi}\label{defi: scsep}
$A \in \GA$ is said to have the stable cone seminorm extension property if the morphism
\[ A \otimes X \xrightarrow{i_1} (A \otimes X) \owedge I\]
has the seminorm extension property for all pointed compact spaces $X$.
\end{defi}
Note that, in light of \cref{rem: main rem}.2, not all $A \in \GA$ have the stable cone seminorm extension property. However, we have the following result.
\begin{prop}\label{prop: stable cone seminorm extension}
    For any separable $G$-$C^*$-algebra $A$, $\mathfrak{q}A \boxtimes \hatkg$ has the stable cone seminorm extension property.
\end{prop}
\begin{proof}
    This is a direct equivariant analogue of \cite[Proposition 7.12]{MJ}, so we don't provide a proof. We mention that we need to use \cref{cor: stable cone seminorm on B(H)} on the $KK^G$-contractibility of $C(X,\mathcal{B}(\mathcal{H}_G)^c)$ here.
\end{proof}
\section{Model categories representing equivariant KK-theory}
We are now in a position to obtain the model structure on $\GA$ representing equivariant KK-theory. The proofs in this section hinge on the several technical results we obtained in the last few sections. We first state a general result from the theory of model categories. It can be proved similar to \cite[Theorem 11.3.2]{HIR}, cf. \cite[Proposition 8.7]{JOHN}.
\begin{prop}\label{prop: abstract lifting}
    Let $B$ be a cofibrantly generated model category with $I$ and $J$ as its set of generating cofibrations and generating acyclic cofibrations, respectively. Let $C$ be a bicomplete category. Suppose $\Sigma$ is a set, and we have a family of adjunctions
    \begin{tikzcd}
B
\arrow[r, "F_s"{name=F}, bend left=25] &
C
\arrow[l, "U_s"{name=G}, bend left=25]
\arrow[phantom, from=F, to=G, "\dashv" rotate=-90]
\end{tikzcd}, $\forall~ s \in \Sigma$. Define the following sets:
\begin{align*}
    &\tilde{I}:=\{F_sg \mid g \in I, s \in \Sigma\},~\tilde{J}:=\{F_sg \mid g \in J, s \in \Sigma\}.
\end{align*}
Assume that the domains of $\tilde{I}$ (respectively $\tilde{J}$) are $\kappa$-small with respect to $\tilde{I}$-(respectively $\tilde{J}$-) cells for some cardinal $\kappa$, and that $U_sf$ is a weak equivalence for every $s \in \Sigma$ whenever $f \in \tilde{J}\operatorname{-cof}$. 

Then, there exists a cofibrantly generated model structure on $C$ with $\tilde{I}$ and $\tilde{J}$ as the sets of generating cofibrations and acyclic cofibrations, respectively, and where a map $f$ is a fibration (respectively weak equivalence) if and only if $U_sf$ is a fibration (respectively weak equivalence) $\forall~s \in \Sigma$.
\end{prop}
\begin{cor}
\label{cor: simp model}    Suppose that in the setup of \cref{prop: abstract lifting}, $B,C$ are simplicial categories which are tensored and cotensored over $\operatorname{sSet}$, and the functors $F_s,U_s$ form a simplicial adjoint pair. Furthermore, suppose that the $\operatorname{sSet}$ cotensoring of $C$ is taken by each $\operatorname{U_s}$ to the $\operatorname{sSet}$ cotensoring of $B$, that is, \[\forall s \in \Sigma, U_s(c^{k_1\rightarrow k_2}) = U_s(c)^{k_1\rightarrow k_2}.\] Then, if $B$ is a simplicial model category, then the transferred model structure and the $\operatorname{sSet}$-enrichment on $C$ are compatible, and make $C$ into a simplicial model category. 
\end{cor}
\begin{proof}
    This essentially follows from the proof of \cite[Proposition 4.2]{transfer}; the family of adjunctions case doesn't pose any additional difficulty, just as in \cref{prop: abstract lifting}.
\end{proof}
Before we proceed, we remind the reader of the finite generation of the standard model structure on $\operatorname{sSet}$ (refer to \cite{HOV} for details): the set of maps \[
I:=\{\partial\Delta^n\hookrightarrow \Delta^n \mid n\geq 0\}, J:=\{\Lambda_i^n \hookrightarrow \Delta^n \mid 0 \leq i \leq n, n > 0\}
\]
are a set of generating cofibrations and acyclic cofibrations, respectively, for the model structure. We also note here a consequence of the fact that $\operatorname{sSet}$ is in fact a simplicial model category (refer to \cite[Section 4.2]{HOV}), which will be relevant later.

\begin{lem}\label{lem: fib induced by cofib}
    Let $X \rightarrow Y$ be a cofibration in $\operatorname{sSet}$. Then, for any $Z \in \operatorname{sSet}$, the induced map \[\operatorname{Map}_{\operatorname{sSet}}(Y,Z) \rightarrow \operatorname{Map}_{\operatorname{sSet}}(X,Z)\] is a fibration.
\end{lem}

We now state our main result regarding model structures representing equivariant KK-theory. 
\begin{prop}\label{thm: main}
    Let $\mathfrak{A}$ denote the set of isomorphism classes of separable $G$-$C^*$-algebras. Consider the sets
    \[
    \tilde{I}:=\{A \otimes \lvert g \rvert \mid g \in I, A=\mathfrak{q}A' \boxtimes \hatkg~\text{where}~A' \in \mathfrak{A}\}, \tilde{J}:=\{A \otimes \lvert g \rvert \mid g \in J, A=\mathfrak{q}A' \boxtimes \hatkg~\text{where}~A' \in \mathfrak{A}\}. 
    \]
    Then, there exists a cofibrantly generated simplicial model structure on $\GA$ with $\tilde{I}$ and $\tilde{J}$ as the sets of generating cofibrations and acyclic cofibrations, respectively, and where a map $f:B\rightarrow C$ is a fibration (respectively weak equivalence) if and only if \[f_*:\operatorname{Hom}(\mathfrak{q}A' \boxtimes \hatkg,B)_\bullet \rightarrow \operatorname{Hom}(\mathfrak{q}A' \boxtimes \hatkg,C)_\bullet\] is a fibration (respectively weak equivalence) for all separable $G$-C*-algebras $A'$.
\end{prop}
Before we proceed to prove the theorem, we note a few observations which will come handy later. We follow the notation of \cref{thm: main} in the rest of this section.
\begin{rem}\label{rem: for stability}
    \hfill
    \begin{enumerate}
        \item Note that, \cref{thm: main} equivalently asserts that we can transfer the simplicial model structure on $\operatorname{sSet}$ along the adjunctions of \cref{cor: adjunction}, one for each $A=\mathfrak{q}A' \boxtimes \hatkg$, where $A' \in \mathfrak{A}$.
        \item In light of \cref{lem: uli} and \cref{cor: kk_g eq}, a map $f: B \rightarrow C$ in the model structure prescribed in \cref{thm: main} is a weak equivalence (resp. fibration)
    if and only if
    \[f_*:\operatorname{Hom}(\mathfrak{q}A',B \otimes \kg)_\bullet \rightarrow \operatorname{Hom}(\mathfrak{q}A' ,C\otimes \kg)_\bullet\] is a $\pi_*$-isomorphism (resp. fibration) for each $A'$. 
    \item For each $B \in \GA$, the map $B \otimes \kg \xrightarrow{1_B \otimes \alpha} B \otimes \hatkg$ is a weak equivalence in $\GA$. Indeed, we have to show that for each $A'$, the induced map
    \[\operatorname{Hom}(\mathfrak{q}A' \boxtimes \hatkg,B \otimes \kg)_\bullet \xrightarrow{\alpha}\operatorname{Hom}(\mathfrak{q}A' \boxtimes \hatkg,B \otimes \hatkg)_\bullet
    \]
    is a weak equivalence. In light of \cref{rem: for stability}, it suffices to show that the adjoint map
    \[
    \operatorname{Hom}(\mathfrak{q}A',B \otimes \kg\otimes\hatkg)_\bullet \xrightarrow{\alpha_*} \operatorname{Hom}(\mathfrak{q}A',B \otimes \hatkg \otimes \hatkg)
    \]
    is a weak equivalence. But this follows from \cref{lem: uli}.
    \item The adjunctions \[
    \begin{tikzcd}
\GA
\arrow[r, "- \boxtimes \kg"{name=F}, bend left=25] &
\GA
\arrow[l, "- \otimes \kg"{name=G}, bend left=25]
\arrow[phantom, from=F, to=G, "\dashv" rotate=-90]
\end{tikzcd}
    \]
    in fact form a Quillen adjunction. Indeed, in light of \cite[Lemma 1.3.4]{HOV}, it suffices to show that $- \otimes \kg: \GA \rightarrow \GA$ preserves fibrations and acylic fibrations. This follows directly from the definition of (acyclic) fibrations, \cref{prop: hom class}, and from the universal absorption property of $\kg$.
    \end{enumerate} 
\end{rem}
The proof of \cref{thm: main} will be executed via a series of lemmas. In what follows, we use the notation of \cref{thm: main}. We also use $\R_A(-)$ to denote the functor $\operatorname{Hom}(A,-)_\bullet$.
\begin{lem} \label{lem: gen cof SEP}
    Morphisms in $\tilde{I}$ and $\tilde{J}$ have the seminorm extension property.
\end{lem}
\begin{proof}
    We begin with the case of $\tilde{J}$. This follows upon noting that $\lvert \Lambda_i^n\rvert \hookrightarrow \lvert \Delta^n\rvert$ is the inclusion of a retract, in light of \cref{example 2}.2. 

    The case of $\tilde{I}$ is a lot more complicated, and hinges upon several results we proved earlier. First, we note that a direct application of the adjunction in \cref{prop: partial adjunction unpointed} yields that every $A \otimes \lvert g \rvert \in \tilde{I}$ has the left lifting property with respect to all maps $p$ for which $\R_A(p)$ is an acyclic fibration. In particular, we have a lift in the following square (in light of \cref{rem: main rem}.3):
    \[
\begin{tikzcd}
	{A \otimes \lvert \Lambda_i^n \rvert} & {(A \otimes \lvert \Lambda_i^n \rvert) \owedge I} \\
	{A \otimes \lvert \Delta^n \rvert} & 0
	\arrow["{i_1}", from=1-1, to=1-2]
	\arrow["{A \otimes \lvert g \rvert}"', from=1-1, to=2-1]
	\arrow[from=1-2, to=2-2]
	\arrow[dashed, from=2-1, to=1-2]
	\arrow[from=2-1, to=2-2]
\end{tikzcd}
    \]
    The map $i_1$ has the seminorm extension property from \cref{prop: stable cone seminorm extension}. The rest follows from \cref{lem: JJ07}.
\end{proof}
\begin{lem}\label{lem: smallness of I and J tilde}
    There exists a cardinal $\kappa$ such that $\tilde{I}$(respectively $\tilde{J}$) is $\kappa$-small with respect to $\tilde{I}$-(respectively $\tilde{J}$-)cells.
\end{lem}
\begin{proof}
    We just discuss the case of $\tilde{I}$; the case of $\tilde{J}$ is similar. By definition, every $\tilde{I}$-cell complex is a transfinite composition of pushouts of elements of $I$ (see \cite[Definition 2.1.9]{HOV}). Since every morphism in $\tilde{I}$ has the seminorm extension property by \cref{lem: gen cof SEP}, and pushouts and transfinite compositions of morphisms having the seminorm extension property have the seminorm extension property by \cref{lem: JJ08} and \cref{lem: technical 3}, we conclude that every $\tilde{I}$-cell complex has the seminorm extension property. The rest follows from \cref{prop: small wrt sep}.
\end{proof}
\begin{lem}\label{lem: verification 2 of abstract lifting}
    Let $g: B \rightarrow D$ be a morphism in $\GA$ which belongs to $\tilde{J}\operatorname{-cof}$. Then, $\R_A(g)$ is a weak equivalence for every $A=\mathfrak{q}A' \boxtimes \hatkg$, where $A' \in \mathfrak{A}$.
\end{lem}
\begin{proof}
    We will in fact show that $g$ is a homotopy equivalence in $\GA$. By definition of $\tilde{J}\operatorname{-cof}$ and the adjunction of \cref{prop: partial adjunction unpointed}, we conclude that $g$ has the left lifting property with respect to all morphisms $p$ for which $\R_A(p)$ is a fibration for each $A=\mathfrak{q}A' \boxtimes \hatkg$, with $A' \in \mathfrak{A}$. Since Kan complexes are precisely the fibrant objects in $\operatorname{sSet}$, in light of \cref{prop: Kan enriched}, the terminal map $B \rightarrow 0$ is such a morphism. In light of \cref{prop: tensored and cotensored} and \cref{lem: fib induced by cofib}, we conclude that the map
    \[
    p: C(\lvert \Delta^1 \rvert,D) \rightarrow C(\lvert \partial\Delta^1 \rvert,D),
    \]
    induced by the inclusion $\lvert \partial\Delta^1 \rvert \hookrightarrow \lvert \Delta^1 \rvert$, is also such a morphism.

    We thus have a lift in the diagram
    \[
\begin{tikzcd}
	B & B \\
	D & 0
	\arrow["{=}", from=1-1, to=1-2]
	\arrow["g"', from=1-1, to=2-1]
	\arrow[from=1-2, to=2-2]
	\arrow["r"{description}, dashed, from=2-1, to=1-2]
	\arrow[from=2-1, to=2-2]
\end{tikzcd},
    \]
    and also in the diagram
    \[
\begin{tikzcd}
	B && {C(\lvert \Delta^1 \rvert,D)} \\
	D && {C(\lvert \partial\Delta^1 \rvert,D)}
	\arrow["{b \mapsto(t \mapsto g(b))}", from=1-1, to=1-3]
	\arrow["g"', from=1-1, to=2-1]
	\arrow[from=1-3, to=2-3]
	\arrow[dashed, from=2-1, to=1-3]
	\arrow["{(g \circ r, \operatorname{id})}"', from=2-1, to=2-3]
\end{tikzcd}.
    \]
    It follows that $g$ and $r$ are mutual homotopy inverses.
\end{proof}
\begin{proof}[Proof of \cref{thm: main}]
    \cref{lem: smallness of I and J tilde} and \cref{lem: verification 2 of abstract lifting} show that the hypothesis necessary to apply \cref{prop: abstract lifting} hold in our set up. The existence of the cofibrantly generated model structure then follows from \cref{prop: abstract lifting}. That this is indeed a simplicial model structure follows from \cref{cor: cotensor preserved} and \cref{cor: simp model}.
\end{proof}
\begin{cor}\label{cor: weak eq in mod cat}
    A map $A \rightarrow B$ between separable $G$-$C^*$-algebras is a weak homotopy equivalence in our model structure if and only if it is a $KK^G$-equivalence.
\end{cor}
\begin{proof}
    Immediate from \cref{cor: kk equi} and the Yoneda lemma.
\end{proof}
We now identify path objects in the sense of model categories (\cite[Definition 1.4.2.2]{HOV}) in $\GA$. This will come handy when we cast the theorem in a form which can be expressed in the classical equivariant $KK$-theory language.
\begin{lem}\label{lem: path object}
    Given any $B \in \GA$, $\C(I,B)$ and $\C_0(\R,B)$ are a path object and a based loop object for $B$, respectively.
\end{lem}
\begin{proof}
    As in the proof of \cref{lem: verification 2 of abstract lifting}, we have that the evaluation at end-points map 
    \[
    C(I,B) \rightarrow B \times B
    \]
    is a fibration in $\GA$. We also have that inclusion $B \rightarrow \C(I,B)$ is a homotopy equivalence of algebras. This shows that $\C(I,B)$ is a path object for $B$. The other assertion follows from the identification $\C_*(I,B) \cong \C(I,B) \times_{B \times B}0$.
\end{proof}
We next want to understand the cofibrant and fibrant objects in $\GA$. Since Kan complexes are precisely the fibrant objects in $\GA$, in light of \cref{prop: Kan enriched} and \cref{thm: main}, every object in $\GA$ is also fibrant. The following lemma deals with the case of cofibrant objects.
\begin{cor}\label{cor: qA boxtimes hatkg is cofibrant}
    For any separable $G$-$C^*$-algebra $A$, $\mathfrak{q}A \boxtimes \hatkg$ is a cofibrant object in $\GA$, and is homotopic (in the model category sense) to a cofibrant replacement of $A$.
\end{cor}
\begin{proof}
    To show that $\mathfrak{q}A \boxtimes \hatkg$ is a cofibrant object, we have to find a lift in the following diagram, where $B \rightarrow D$ is any acyclic fibration
        \[
\begin{tikzcd}
	0 & B \\
	{\mathfrak{q}A \boxtimes \widehat{\mathcal{K}_G}} & D
	\arrow[from=1-1, to=1-2]
	\arrow[from=1-1, to=2-1]
	\arrow["\simeq", two heads, from=1-2, to=2-2]
	\arrow[dashed, from=2-1, to=1-2]
	\arrow[from=2-1, to=2-2]
\end{tikzcd}
        \]
        The existence of the lift follows immediately from the definition of weak equivalences and fibrations in our model structure, and the fact that an acyclic fibration between Kan complexes is surjective on the $0$-simplices.
        
        We next consider the following maps in $\GA$
        \[
        \mathfrak{q}A \boxtimes \hatkg \xrightarrow{\epsilon^{\operatorname{ad}}} \mathfrak{q}A = q(A \otimes \kg) \otimes \kg \xrightarrow{\pi_0 \otimes \operatorname{id}_\kg} A \otimes \kg \otimes \kg  \xrightarrow{\cong} A \otimes \kg \xrightarrow{} A \otimes \hatkg \xleftarrow{\epsilon} A.
        \]
        Each of the maps above are weak equivalences in $\GA$; $\epsilon^{\operatorname{ad}}$ is a weak equivalence by \cref{prop: B boxtimes K mapsto B is weak eq}, and the rest are $KK^G$-equivalences from the classical theory (see \cite{BLAK} for details). The rest follows from \cref{cor: weak eq in mod cat} and the fact that the canonical map
        \[
            \operatorname{Ho}(\mathcal{M}^b) \rightarrow \mathcal{M}[w^{-1}]
        \]
        is an equivalence for any model category $\mathcal{M}$, where $\mathcal{M}^b$ denotes the full subcategory of its bifibrant objects of $\mathcal{M}$.
\end{proof}
We finally cast \cref{thm: main} in the language of classical equivariant $KK$-theory.
\begin{thm}\label{thm: diff formulation of main}
   There is a simplicial model structure on the category of $\nu$-complete locally multiplicative convex $G$-$C^*$-algebras satisfying the following properties:
    \begin{enumerate}
        \item for a separable $G$-$C^*$-algebra $A$ and $B \in \GA$, there is a natural equivalence 
        \[
        \operatorname{Map}_{\GA[w^{-1}]}(A,B)\simeq KK^G(A,B)_\bullet,
        \]
        where the first entity stands for the mapping anima in the $\infty$-category obtained by inverting the weak equivalences of the model structure. In particular, when $A$,$B$ are separable $G$-$C^*$-algebras, we have a natural isomorphism
        \[
        \operatorname{Ho}(A,B) \cong KK^G_{0}(A,B).
        \]
        \item a $G$-equivariant $*$-homomorphism $A \rightarrow A'$ of separable $G$-$C^*$-algebras is a weak equivalence for the model category structure if and only if it is a $KK^G$-equivalence. It is a fibration if and only if the induced map \[\operatorname{Hom}(q(D\otimes \kg)\otimes \kg,A \otimes \kg)_\bullet \rightarrow \operatorname{Hom}(q(D \otimes \kg) \otimes \kg,A' \otimes \kg)_\bullet\] is a Kan fibration for separable $G$-$C^*$-algebras $D$.
        \item the model structure is cofibrantly generated.
        \item each object is fibrant.
        \item the model structure is stable, and the equivalences in Point 1 refine to equivalences of commutative groups.
    \end{enumerate}
\end{thm}
\begin{proof}
    Point 1 follows from the fundamental theorem of simplicial model categories (\cite[Theorem 1.3.4.20]{HA}), \cref{cor: qA boxtimes hatkg is cofibrant}, \cref{prop: hom class}, together with \cref{cor: datta picture}. Point 2 follows from \cref{cor: weak eq in mod cat} and \cref{rem: main rem}.2. Points 3 and 4 are already discussed above. The first assertion of Point 5 follows from \cref{lem: path object} and \cref{cor: Bott}. For the second assertion, it suffices to recall that every stable $\infty$-category is additive in a canonical way; see \cite[Lemma 1.1.2.9]{HA}.
\end{proof}
\subsection{Comparison with Bunke's $KK^G_{\operatorname{sep}}$}
We end the article by making a comparison of our model category with the recent stable $\infty$-category $KK^G_{\operatorname{sep}}$ à la Bunke, Engel, Land (\cite{BEL}, \cite{BD}). Although \cite{BEL} is written for countable discrete group actions, the results we use from there can be easily verified to also work for a general locally compact, second countable group. For the convenience of the reader, we recall the definition here.
\begin{defi}[\cite{BEL}, Definition 1.2]
    Let $G$ be a locally compact, second countable group. We denote by $KK^G_{\operatorname{sep}}$ the Dwyer-Kan localization of $G\text{-}C^*\operatorname{-alg}$ at the collection of $KK^G$-equivalences.
\end{defi}
We begin by describing some basic features of the $\infty$-category underlying our model category. For the remainder of the section, we denote by $\mathcal{M}$ and $\mathcal{M}_0$ the category $\GA$, and the full subcategory of separable $G$-$C^*$-algebras inside $\GA$, respectively. We also denote by $w$ and $w_0$ the class of weak equivalences in $\mathcal{M}$, and the subclass of weak equivalences whose domain and codomain are in $\mathcal{M}_0$, respectively.
\begin{lem}\label{lem: localize}
    \hfill
    \begin{enumerate}
        \item The $\infty$-category $\mathcal{M}_0[w_0^{-1}]$ is left-exact.
        \item The localization functor $L_{w_0}:\mathcal{M}_0\rightarrow\mathcal{M}_0[w_0^{-1}]$ sends fibrant cartesian squares to cartesian squares.
        \item The induced functor $\overline{L_w \circ \iota}:\mathcal{M}_0[w_0^{-1}] \rightarrow \mathcal{M}[w^{-1}]$ is left-exact.
    \end{enumerate}
\end{lem}
\begin{proof}
    Let $F_0$ denote the class of fibrations in $\mathcal{M}$ whose domain and codomain are in $\mathcal{M}_0$. Then, as an immediate corollary to \cref{thm: main} and \cref{thm: diff formulation of main}, we conclude that $(\mathcal{M}_0,w_0,F_0)$ is a category of fibrant objects in the sense of \cite[Definition 7.4.12]{CIS}. The conclusions of Assertions 1 and 2 then follow from \cite[Proposition 7.5.6]{CIS}. Assertion 3 also follows from loc. cit., upon observing that the functor \[L_w \circ \iota: \mathcal{M}_0 \rightarrow \mathcal{M}[w^{-1}]\] sends fibrant cartesian squares to cartesian squares; $\iota$ tautologically preserves fibrant cartesian squares, and $L_w$ takes fibrant cartesian squares to cartesian sqaures by a variant of Assertion 2 for $\mathcal{M}$.
\end{proof}
\begin{rem}
    \cref{lem: localize} also holds good with $\mathcal{M}_0$ replaced by $\mathcal{M}$. Using ideas in the proof of \cite[Proposition 7.5.6]{CIS}, we can also show that the pullback along $L_{w_0}$ and $L_w$ respectively induce for every left-exact $\infty$-category $D$ equivalences
        \begin{align*}
            L_{w_0}^*:&\operatorname{Fun}^{\operatorname{lex}}(\mathcal{M}_0[w_0^{-1}],D) \xrightarrow{\simeq} \operatorname{Fun}^{w,\operatorname{fib,red}}(\mathcal{M}_0,D). \\
            L_{w}^*:&\operatorname{Fun}^{\operatorname{lex}}(\mathcal{M}[w^{-1}],D) \xrightarrow{\simeq} \operatorname{Fun}^{w,\operatorname{fib,red}}(\mathcal{M},D).
        \end{align*}
\end{rem}
\begin{prop}\label{bunke fully faithful}
    The natural functor
    \[
    KK^G_{\operatorname{sep}} \rightarrow \mathcal{M}[w^{-1}]
    \]
    induced by the inclusion of separable $G$-$C^*$-algebras inside all $\nu$-complete locally multiplicative convex $G$-$C^*$-algebras is fully faithful.
\end{prop}
\begin{proof}
    Consider the following diagram:
    \[
\begin{tikzcd}
	{G\text{-}C^*\text{-alg}_{\operatorname{sep}}} & {\mathcal{M}_0} & {\mathcal{M}} \\
	{KK^G_{\operatorname{sep}}} & {\mathcal{M}_0[w_0^{-1}]} & {\mathcal{M}[w^{-1}]}
	\arrow["\cong", from=1-1, to=1-2]
	\arrow["{L_{KK^G}}"', from=1-1, to=2-1]
	\arrow[hook, from=1-2, to=1-3]
	\arrow["{L_{w_0}}"', from=1-2, to=2-2]
	\arrow["{L_w}"', from=1-3, to=2-3]
	\arrow["\simeq", from=2-1, to=2-2]
	\arrow["F", from=2-2, to=2-3]
\end{tikzcd}.
    \]
    The bottom left arrow is evidently an equivalence by \cref{thm: diff formulation of main}.2. Thus, it suffices to show that $F$ is fully faithful. In light of \cref{lem: localize}.3, it suffices to show F induces an isomorphism on $\pi_0$; we call this a \textit{$\pi_0$-equivalence} in the remainder of the proof. To that end, we consider the following ladder diagram
    \[
\begin{tikzcd}
	{\operatorname{Map}_{KK^G_{\operatorname{sep}}}(A,B)} && {\operatorname{Map}_{\mathcal{M}[w^{-1}]}(A,B)} \\
	{\operatorname{Map}_{KK^G_{\operatorname{sep}}}(A \otimes \mathcal{K}_G,B \otimes\mathcal{K}_G)} && {\operatorname{Map}_{\mathcal{M}[w^{-1}]}(A \otimes \mathcal{K}_G,B \otimes \mathcal{K}_G)} \\
	{\operatorname{Map}_{KK^G_{\operatorname{sep}}}(q(A \otimes \mathcal{K}_G)\otimes \mathcal{K}_G,B \otimes\mathcal{K}_G)} && {\operatorname{Map}_{\mathcal{M}[w^{-1}]}(q(A \otimes \mathcal{K}_G)\otimes \mathcal{K}_G,B \otimes\mathcal{K}_G)} \\
	{\operatorname{Map}_{h}(q(A \otimes \mathcal{K}_G)\otimes \mathcal{K}_G,B \otimes\mathcal{K}_G)} && {\operatorname{Map}_{\mathcal{M}[h_{\mathcal{M}}^{-1}]}(q(A \otimes \mathcal{K}_G)\otimes \mathcal{K}_G,B \otimes\mathcal{K}_G)}
	\arrow["F", from=1-1, to=1-3]
	\arrow["{-\otimes\mathcal{K}_G}"', from=1-1, to=2-1]
	\arrow["{-\otimes\mathcal{K}_G}", from=1-3, to=2-3]
	\arrow[from=2-1, to=2-3]
	\arrow["\alpha", from=2-1, to=3-1]
	\arrow["\alpha", from=2-3, to=3-3]
	\arrow[from=3-1, to=3-3]
	\arrow["{\text{invert}~KK^G\text{-equivalences}(!)}", from=4-1, to=3-1]
	\arrow["{F_h}", from=4-1, to=4-3]
	\arrow["{\text{invert}~w(!!)}"', from=4-3, to=3-3]
\end{tikzcd}.
    \]
    In order to show that $F$ is a $\pi_0$-equivalence, it suffices to show that all the vertical arrows are $\pi_0$-equivalences, and also $F_h$ is an $\pi_0$-equivalence. That the downwards pointing vertical maps are equivalences follows immediately from the classical theory; cf. the proof of \cref{cor: qA boxtimes hatkg is cofibrant}. It thus remains to prove the following:
    \begin{itemize}
        \item $F_h$ is an equivalence.
    \end{itemize}
    \begin{proof}
        This follows from the fact that the Kan-enriched categories $G\text{-}C^*\text{-}\operatorname{alg}$ and $\GA$ directly present the localization at homotopy equivalences, and \cref{cor: fully faithful oo cat.}. 
    \end{proof}
    \begin{itemize}
    \item $(!)$ is an $\pi_0$-equivalence.
\end{itemize}
\begin{proof}
    In light of \cite[Proposition 2.18]{BEL}, we have to show that the induced map
    \[
    [\mathfrak{q}A,B \otimes \kg] \xrightarrow{!} KK^G_{\operatorname{class}}(\mathfrak{q}A,B \otimes \kg) = [\mathfrak{q}^2A,B \otimes \kg \otimes \kg]
    \]
    is an isomorphism. This follows from the results obtained in Section 6.
\end{proof}
\begin{itemize}
    \item (!!) is an equivalence.
\end{itemize}
\begin{proof}
    We consider the follow commutative diagram:
    \[
\begin{tikzcd}
	{\operatorname{Map}_{\mathcal{M}[h_\mathcal{M}^{-1}]}(\mathfrak{q}A, B \otimes \mathcal{K}_G)} & {\operatorname{Map}_{\mathcal{M}[h_\mathcal{M}^{-1}]}(\mathfrak{q}A, B \otimes \widehat{\mathcal{K}_G})} & {\operatorname{Map}_{\mathcal{M}[h_\mathcal{M}^{-1}]}(\mathfrak{q}A \boxtimes \widehat{\mathcal{K}_G}, B)} \\
	{\operatorname{Map}_{\mathcal{M}[w^{-1}]}(\mathfrak{q}A, B \otimes \mathcal{K}_G)} & {\operatorname{Map}_{\mathcal{M}[w^{-1}]}(\mathfrak{q}A, B \otimes \widehat{\mathcal{K}_G})} & {\operatorname{Map}_{\mathcal{M}[w^{-1}]}(\mathfrak{q}A \boxtimes \widehat{\mathcal{K}_G}, B)}
	\arrow["{\alpha_*}", from=1-1, to=1-2]
	\arrow["{(!!)}"', from=1-1, to=2-1]
	\arrow[from=1-2, to=1-3]
	\arrow["{(\dagger)}"', from=1-2, to=2-2]
	\arrow["{(\dagger\dagger)}"', from=1-3, to=2-3]
	\arrow["{\alpha_*}", from=2-1, to=2-2]
	\arrow[from=2-2, to=2-3]
\end{tikzcd}.
    \]
    To show that $(!!)$ is an equivalence, it suffices to show that all the horizontal arrows, as well as $(\dagger\dagger)$ are equivalences. The top left horizontal map is an equivalence from \cref{lem: uli}, whereas the bottom left horizontal map is an equivalence by \cref{rem: for stability}.3. The top right horizontal map is an equivalence by \cref{cor: inf cat adjunction}, whereas the bottom right map is an equivalence by \cref{rem: for stability}.4. $(\dagger \dagger)$ is an equivalence follows by observing that $\mathfrak{q}A \boxtimes \hatkg$ is colocal for the weak equivalences in $\mathcal{M}$ by definition, in light of \cite[Remark 9.9]{BD}.
\end{proof}
This completes the proof of the Proposition.
\end{proof}

\printbibliography
\end{document}